\documentclass[12pt,reqno,tbtags]{amsproc}
\usepackage{a4wide}
\usepackage{amssymb}

\usepackage{amsmath, amsfonts, amsthm, amssymb, amscd}
\usepackage[matrix, arrow, curve]{xy}
\usepackage{latexsym}
\usepackage[dvips]{graphicx}

\DeclareMathOperator{\link}{link} 
 \DeclareMathOperator{\rk}{rk}
\DeclareMathOperator{\cone}{cone} \DeclareMathOperator{\pyr}{pyr}
\DeclareMathOperator{\Tor}{Tor}
\DeclareMathOperator{\colim}{colim}
\DeclareMathOperator{\hocolim}{hocolim}
\DeclareMathOperator{\Ob}{Ob} \DeclareMathOperator{\pt}{pt}
\DeclareMathOperator{\sta}{star}

\newcommand{\zk}{\mathcal{Z}_K}
\newcommand{\zp}{\mathcal{Z}_P}
\newcommand{\zg}{\mathcal{Z}_G}
\newcommand{\z}{\mathcal{Z}}
\newcommand{\zz}{\mathcal{Z}}

\newcommand{\Zo}{\mathbb{Z}}
\newcommand{\Ro}{\mathbb{R}}
\newcommand{\Co}{\mathbb{C}}

\newcommand{\F}{\mathcal{F}}

\newcommand{\ccal}{\mathcal{C}}
\newcommand{\acal}{\mathcal{A}}
\newcommand{\bcal}{\mathcal{B}}
\newcommand{\ncal}{\mathcal{N}}

\newcommand{\cat}{\mathrm{cat}}
\newcommand{\Top}{\mathrm{Top}}

\newcounter{examcounter}[section]
\newcounter{stmcounter}[section]

\newcounter{defcounter}[section]
\newcounter{problcounter}

\renewcommand{\theexamcounter}{\thesection.\arabic{examcounter}}

\newcommand{\ex}{\par\vspace{0.5 cm}\noindent\refstepcounter{examcounter}\textsc{Example \theexamcounter.}\quad}
\newcommand{\rem}{\par\vspace{0.5 cm}\noindent\textsc{Remark.}\quad}

\newtheorem{cor}[stmcounter]{Corollary}

\newtheorem{theorem}[stmcounter]{Theorem}
\newtheorem{prop}[stmcounter]{Proposition}
\newtheorem{lemma}[stmcounter]{Lemma}
\newtheorem{defin}[defcounter]{Definition}
\newtheorem{problem}[problcounter]{Problem}

\usepackage[in]{fullpage}

\begin{document}

\title{Moment-angle complexes and polyhedral products \\ for convex polytopes.}%
\author{A.~A.~Ayzenberg, V.~M.~Buchstaber}%
\address{Department of Mechanics and Mathematics, Moscow State University, Moscow, Russia}%
\email{ayzenberga@gmail.com}%
\address{V.A. Steklov Mathematical Institute, RAS, Moscow, Russia}%
\email{buchstab@mi.ras.ru}%

\begin{abstract}
Let $P$ be a convex polytope not simple in general. In the focus
of this paper lies a simplicial complex $K_P$ which carries
complete information about the combinatorial type of $P$. In the
case when $P$ is simple, $K_P$ is the same as $\partial P^*$,
where $P^*$ is a polar dual polytope. Using the canonical
embedding of a polytope $P$ into nonnegative orthant
$\Ro_{\geqslant}^m$, where $m$ is a number of its facets, we
introduce a moment-angle space $\zp$ for a polytope $P$. It is
known, that in the case of a simple polytope $P$ the space $\zp$
is homeomorphic to the moment-angle complex $(D^2,S^1)^{K_P}$.
When $P$ is not simple, we prove that the space $\zp$ is
homotopically equivalent to the space $(D^2,S^1)^{K_P}$. This
allows to introduce bigraded Betti numbers for any convex
polytope. A Stanley-Reisner ring of a polytope $P$ can be defined
as a Stanley-Reisner ring of a simplicial complex $K_P$. All these
considerations lead to a natural question: which simplicial
complexes arise as $K_P$ for some polytope $P$? We have proceeded
in this direction by introducing a notion of a polytopic
simplicial complex. It has the following property: link of each
simplex in a polytopic complex is either contractible, or
retractible to a subcomplex, homeomorphic to a sphere. The complex
$K_P$ is a polytopic simplicial complex for any polytope $P$.
Links of so called face simplices in a polytopic complex are
polytopic complexes as well. This fact is sufficient enough to
connect face polynomial of a simplicial complex $K_P$ to the face
polynomial of a polytope $P$, giving a series of inequalities on
certain combinatorial characteristics of $P$. Two of these
inequalities are equalities for each $P$ and represent
Euler-Poincare formula and one of Bayer-Billera relations for flag
$f$-numbers. In the case when $P$ is simple all inequalities turn
out to be classical Dehn-Sommerville relations.
\end{abstract}

\maketitle

\section{Introduction}

This work is devoted to the application of the theory of
moment-angle complexes developed for simple polytopes and
simplicial complexes to the study of nonsimple convex polytopes.

Let $P$ be a convex polytope with $m$ facets. There is a classical
construction of convex geometry which allows to associate to each
such polytope a dual polar polytope $P^*$. It can be defined
geometrically, by using the construction of polar set in a
euclidian space. The crucial property is the fact that the
partially ordered set (poset) of faces of $P^*$ coincides with the
poset of faces of $P$ with the reversed order (see details in
\cite{Zieg}).

In the case when $P$ is a simple polytope of dimension $n$ (that
is any vertex is contained in exactly $n$ facets) the dual
polytope $P^*$ is simplicial (that is all its facets are
simplices). So one can consider the boundary $\partial P^*$ as a
geometrical realization of an abstract simplicial complex which is
often denoted $K_P$ in a literature
(\cite{BP},\cite{DJ},\cite{Pan}). The simplicial complex $K_P$
carries complete information about combinatorial type of a simple
polytope $P$. Switching between a simple polytope $P$ and
simplicial complex $K_P$ corresponding to this polytope allows to
get different results in toric topology. We now briefly observe
one example since it will make sense for the topic of our paper
(all definitions will be given in the text).

The moment-angle complex $\zp$ of a simple polytope $P$ can be
defined in many different ways. There are two constructions of
moment-angle complex which give homeomorphic topological spaces
for a simple polytope $P$. One can use a geometrical embedding of
$P$ into the nonnegative orthant $\Ro_{\geqslant}^m$ to define
$\zp$ as a pullback of the moment-angle map under such inclusion
\cite{BP2}. Such definition allows to define the structure of a
smooth manifold on $\zp$. On the other hand, $\zp$ can be defined
as a moment-angle complex $(D^2,S^1)^{K_P}$. Such a
characterization appeared in \cite{BP2}, \cite{Bask} and allowed
to calculate cohomology ring of $\zp$. The term "polyhedral
product" was used in \cite{BBCG}, where the construction of a
polyhedral product of the set of pairs $(X_i,A_i), i=1,\ldots,m$
was considered and where the results were obtained about
homotopical properties of these spaces. The space $(X,A)^K$ can be
viewed as a particular case of a polyhedral product. We will use
the term "polyhedral product" for spaces $(X,A)^K$ in the paper
and the term "moment-angle complex" for its particular case
$(D^2,S^1)^K$.

As we see, both definitions of a moment-angle complex of a simple
polytope can be used for different purposes.

Note that for a simple polytope $P$ the complex $K_P$ can be
defined without using dual polytope $P^*$. Indeed, let
$\{\F_1,\ldots,\F_m\}$ be facets of $P$. These facets cover the
set $\partial P$. Taking the nerve of this cover, we get exactly
$K_P$. Such observation can be taken as a definition of $K_P$.

For the sake of simplicity we do not distinguish between abstract
simplicial complexes and their geometrical realizations. The
simplicial complex hereafter means a finite combinatorial object
as well as a topological space. When we write that some simplicial
complex is homotopically equivalent to the space we mean that its
geometrical realization is homotopically equivalent to this space.
When we establish the equality of two complexes we mean that they
are isomorphic as combinatorial objects. We hope that this
agreement will not lead to misunderstanding.

In the case of general convex polytopes $P$ the space $\partial
P^*$ is not a simplicial complex, therefore it does not coincide
with $K_P$ in any sense. So the study of $K_P$ for nonsimple
polytopes $P$ differs from the study of dual polytopes. But even
in this general case the complex $K_P$ have many nice properties
which we would like to discuss in this paper. An interesting thing
is that the complex $K_P$ carries a specific structure which we
call \textit{polytopic simplicial complex} (see definition
\ref{definScomplex}).

Some of these properties generalize results about simple
polytopes. We will try to underline connections with classical
theory where it is possible. In particular, there had been found a
formula for the $f$-polynomial of $K_P$ which gives
Dehn-Sommerville relations in the case when $P$ is simple. The
technique used to prove it is borrowed from theory of
two-parametric face polynomials and the ring of polytopes,
introduced in \cite{Buch}. There is a key point in the original
theory of ring of polytopes --- the commutation property
$F(dP)=\frac{\partial}{\partial t}F(P)$, which allows to express a
face polynomial of $P$ by face polynomials of its facets. This
property allows to deduce Dehn-Sommerville relations from
Euler-Poincare formula for a polytope. Big part of this theory can
be generalized to deal with simplicial complex $K_P$. We introduce
\textit{two-dimensional face polynomial} of a convex polytope $P$
which is a combinatorial invariant of $P$. The coefficients of
this polynomial are subject to a series of certain inequalities.
Two of these inequalities turn out to be equalities for each
convex polytope. One of them represents Euler-Poincare formula for
a polytope, which is not very surprising. But another equality
represents a certain relation on flag $f$-numbers of a polytope.
The latter is one of Bayer-Billera relations \cite{BB}.

The last part of the work is devoted to the properties of
moment-angle spaces for general convex polytopes. As we already
mentioned, the study of moment-angle complexes and polyhedral
products of simplicial complexes is well developed. We tried to
adapt this theory to deal with moment-angle spaces of nonsimple
polytopes. Theorem \ref{theoremEquiv} allows to switch between a
moment-angle space of a polytope $P$ and a moment-angle complex of
a corresponding simplicial complex $K_P$ even when $P$ is not
simple. This allows to introduce Betti numbers of a convex
polytope.

We also underline that moment-angle complex is defined not only
for simplicial complex but for any \textit{hypergraph} also. It is
natural to consider such moment-angle complexes when dealing with
polytopes, though they give nothing new in a topological setting.

The structure of the paper is the following. In section
\ref{Secdefinit} we define a characteristic hypergraph $G_P$ of a
polytope $P$ and a simplicial complex $K_P$. These constructions
carry complete information about combinatorial type of a polytope
$P$. This means that a partially ordered set of facets of $P$ can
be uniquely determined by $K_P$ or $G_P$. We also observe an
operation of a simplicial closure of an arbitrary hypergraph. In
terms of this operation $K_P$ is a simplicial closure of $G_P$.

In section \ref{SecConn} we review the connection of $K_P$ with
the dual polytope $P^*$. It seems reasonable to consider $K_P$ as
some sort of simplicial resolution of $P^*$. We would like to
underline the crucial idea: for nonsimple polytopes simplicial
complex $K_P$ does not coincide with $\partial P^*$ in any sense.


There are connections of $K_P$ with the structure of partially
ordered set of faces of $P$. The complex $K_P$, as every
simplicial complex, carries the structure of partially ordered
set: its simplices are ordered by inclusion. Also the faces of $P$
are partially ordered by inclusion. The connection between these
posets is not the same as the connection between posets of faces
of $P$ and $P^*$. Indeed, the partially ordered set of faces of a
polytope $P^*$ is the same as the set of faces of a polytope $P$
with reversed order. The situation in the case of $K_P$ is
different. In section \ref{SecStructKP} we observe some
combinatorial properties of $K_P$. In particular, we describe how
to restore a structure of $P$ from the structure of $K_P$.

We have found one necessary condition for a simplicial complex to
be $K_P$ for some $P$. It is interesting that this condition leads
to some new class of simplicial complexes. This is the class of so
named \textit{polytopic simplicial complexes}. The exact
definition will be given in section \ref{SecStructKP}, but,
roughly speaking, the polytopic complex is a simplicial complex in
which all links are either contractible or homotopically
equivalent to a sphere. It means that polytopic complex looks
rather like a manifold with a boundary but in homotopical sense.
We prove that $K_P$ is a polytopic complex for any convex polytope
$P$ and extract the information about $f$-vector of $K_P$ using
this structure.

The rest of the work concerns moment-angle complexes for different
objects. Two types of moment-angle complexes are defined: a
moment-angle space $\zp$ of a polytope and a moment-angle complex
$\z_G$ of a hypergraph. The first construction \cite{BP2} uses the
canonical embedding of $P$ into nonnegative orthant
$\Ro_{\geqslant}^m$. The second construction \cite{BP2}
generalizes a notion of moment-angle complex $(D^2,S^1)^K$ for a
simplicial complex $K$. In the case when $P$ is simple, the
moment-angle space of $P$ is homeomorphic to a moment-angle
complex of $K_P$ as we already mentioned. But in general case
these two spaces are not homeomorphic.

Nevertheless these spaces are very similar for different reasons.
First of all, both $\zp$ and $\z_{K_P}$ are equipped with an
action of a torus. For any space $X$ with an action of a torus one
can consider an invariant $s(X)$ which is a maximal rank of toric
subgroups acting freely on this space. In the case of moment-angle
spaces of polytopes (or moment-angle complexes of simplicial
complexes) arises an action of $m$-torus. Then the number $s$ is a
combinatorial invariant of a polytope (resp., simplicial complex).
This invariant of polytopes and simplicial complexes was called
the Buchstaber number and was considered in
\cite{Izm},\cite{Er},\cite{FM},\cite{Ay}. We will show that
$s(\zp)=s(\z_{K_P})$ for any convex polytope $P$ in section
\ref{Secbuch}.

Another reason why $K_P$ is a nice substitute of $P$ is that
$\z_{K_P}$ is homotopically equivalent to $\zp$. We prove this
fact in section \ref{Secmomang} using two other definitions of
$\zp$ which are equivalent to the original one. First definition
describes $\zp$ as an identification space of $P\times T^m$ by
some equivalence relation. This is similar to the situation with
simple polytopes (see \cite{BP}). Another definition uses certain
homotopy colimit \cite{PR}.

Since $\zp\simeq\z_{K_P}$ we may consider an additional grading in
the ring $H^*(\zp,\Zo)$ as well as Betti numbers
$\beta^{-i,2j}(P)$ of a polytope $P$. The Hochster formula
$$
\beta^{-i,2j}(P)=\sum\limits_{\omega\subseteq[m],
|\omega|=j}\dim_{\Zo} \tilde{H}^{j-i-1}(F_\omega;\Zo),
$$
where $F_\omega = \bigcup\limits_{i\in\omega} \F_i\subseteq P$ is
a union of facets, also makes sense in the case of non-simple
polytopes. For the information on Betti numbers and moment-angle
complexes see \cite{Bask},\cite{BP},\cite{Hoch},\cite{Pan}.

So forth Betti numbers and cohomology of $\zp$ can be used as
combinatorial invariants of a polytope $P$.

There are though important differences between moment-angle spaces
for simple and nonsimple polytopes. For example, the restriction
of the moment-angle map over the facet of a polytope cannot be
described by a moment-angle complex over the facet as it is in the
case of simple polytope. See details in example
\ref{examplRestrict}.

As a byproduct of our study we constructed some new invariants of
convex polytopes. This may help in distinguishing combinatorial
types of polytopes. There are many open problems in this area, for
example:

\begin{problem}\label{problbuch}
Let $P$ be a convex polytope such that $P=P^*$. What can we say
about the ring $H^*(\zz_{K_P})$ and invariant $s(P)$ for such $P$?
\end{problem}

The answer may help in understanding the properties of self-dual
polytopes.

Authors are grateful to Taras Panov and Sergey Melikhov for their
valuable remarks and attention to the work.

\section{Basic definitions and constructions}\label{Secdefinit}
%
%

%
%

Let us fix a finite set $[m] = \{1,\ldots,m\}$. We will use a
notion of a hypergraph in our paper. By definition, a hypergraph
is an arbitrary system of subsets of $[m]$, $G\subseteq 2^{[m]}$.
Elements of $G$ are called hyperedges.

A simplicial complex is an important special case of a hypergraph.
By definition, a simplicial complex is such a hypergraph $K$ that:
if $\sigma\in K$ and $\tau\subset \sigma$, then $\tau \in K$.

It is usual in the literature to add a condition that all
singletons lie in $K$. But we will not need this assumption.

Let $G$ be a hypergraph. Let $K_G\subseteq 2^{[m]}$ be the minimal
simplicial complex, such that $G\subseteq K_G$. Such a simplicial
complex $K_G$ is called a simplicial closure of $G$.

Other objects under consideration are polytopes. Let $P=\{
x\in \mathbb{R}^n\mid \langle a_i,x\rangle+b_i\geqslant 0 \}$ be a
convex polytope, where\, $a_i\in \mathbb{R}^n,\; i=1,\ldots,m$ are
the rows of $(m\times n)$-matrix $A$. We obtain the embedding
\[ j_P \colon P \longrightarrow \mathbb{R}^m_{\geqslant}\,:\, j_P(x)=(y_1,\ldots,y_m)\;
\text{ where } y_i=\langle a_i,x\rangle+b_i. \] Using $j_P$ we
will consider $P$ as a polytope in $\mathbb{R}^m$.

For a convex polytope $P\subset\Ro^n$ its polar set $P^*\subset
(\Ro^n)^*$ is defined as $P^*=\{l\in(\Ro^n)^*\mid \langle
l,x\rangle\geqslant -1$ for any $x\in P\}$. If $\dim P = n$ and
the origin lies inside $P$, then $P^*$ is a polytope. By polar
(dual) polytope we mean $P^*$ in this case. The partially ordered
set of its faces coincides with the partially ordered set of faces
of $P$ with reversed order. See \cite{Zieg} for basic information
on polytopes and polar duality.

An $n$-dimensional polytope $P$ is called simple if any vertex of
$P$ is contained in exactly $n$ facets. A polytope is called
simplicial if its facets are simplices. The polar to a simple
polytope is a simplicial polytope.


 Take the function $\sigma \colon \mathbb{R}^m \to
2^{[m]}\,:\, \sigma(y)=\{ i\in [m]\,:\, y_i=0 \}$. The image of
$\sigma$ is the set of all subsets of set $[m]=\{ 1,\ldots,m \}$.
We may restrict function $\sigma$ to the image of a polytope in
$\Ro^m$. This gives a function $\tilde{\sigma}=\sigma\circ
j_p\colon P\to 2^{[m]}$. This function may be characterized
another way. Let $\F_1,\ldots,\F_m$ be facets of $P$ given by
equations $\F_i = \{x\in P\mid \langle a_i,x\rangle+b_i= 0\}$.
Therefore we can identify elements of $[m]$ with facets of a
polytope in an obvious way. Then for a point $x\in P$ we have
$\tilde{\sigma}(x) = \{\i\mid x\in \F_i\}$.

The image of a function $\tilde{\sigma}$ is a system of subsets of
$[m]$, in other words, a hypergraph. We denote this hypergraph by
$G_P$. Taking the simplicial closure of this hypergraph, we obtain
a simplicial complex $K_P$. These two objects provide an important
information about convex polytope.

\begin{prop}
The simplicial complex $K_P$ can be characterized as follows. A
subset $\{i_1,\ldots,i_k\}$ is a simplex of $K_P$ whenever
$\F_{i_1}\cap\ldots\cap\F_{i_k}\neq\varnothing$ in a polytope $P$.
\end{prop}

The proof follows from definitions. In other words the simplicial
complex $K_P$ may be thought of as a nerve of the cover in sense
of Pavel Alexandrov. The boundary of a polytope is covered by its
facets. Then, by definition, the nerve of this cover is $K_P$. An
important thing to note is the following: the dimension of $K_P$
can be arbitrarily large even when the dimension $n$ of $P$ is
fixed. In general, we have $n-1\leqslant\dim K_P\leqslant m-2$ and
this inequality cannot be improved. An example of a polytope $P$
with $\dim K_P = m-2$ is a pyramid (see ex. \ref{examplPyramid} in
section \ref{Secbuch}). One can easily prove that $\dim K_P=n-1$
if and only if a polytope $P$ is simple.

\begin{figure}[h]
\begin{center}
\includegraphics[scale=0.3]{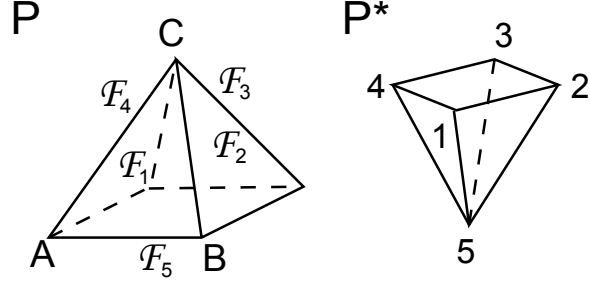}
\end{center}
\caption{Pyramid over a square and its dual polytope}\label{pyr}
\end{figure}

\ex\label{examplPyrSq} Let $P$ be a pyramid over a square shown on
figure \ref{pyr} with a fixed enumeration of facets. The values of
a function $\tilde{\sigma}$ are the following subsets:
$\varnothing$
--- is the image of $\tilde{\sigma}$ on the interior of $P$; $
\{1\},\{2\},\{3\},\{4\},\{5\}$ --- correspond to facets of $P$;
$\{1,2\}$, $\{2,3\}$, $\{3,4\}$, $\{1,4\}$, $\{1,5\}$, $\{2,5\}$,
$\{3,5\}$, $\{4,5\}$ --- correspond to edges of the pyramid and
$\{1,2,5\}$, $\{2,3,5\}$, $\{3,4,5\}$, $\{1,4,5\}$, $\{1,2,3,4\}$
--- these are the values $\tilde{\sigma}$ on vertices of $P$.
We see that these subsets do not form a simplicial complex since
$\{1,2,3,4\}\in G_P$, but $\{1,2,3\}\notin G_P$. But we may
consider a simplicial closure $K_P$ of $G_P$. Its maximal
simplices are
$\{1,2,5\},\{2,3,5\},\{3,4,5\},\{1,4,5\},\{1,2,3,4\}$. The complex
$K_P$ has dimension $3$. It is not pure in this case.

The figure \ref{pyr} also shows the polar polytope to a pyramid,
which is actually a pyramid itself. We see that polar polytope is
an object different from $K_P$ in any sense.

\rem All the combinatorial structure of $P$ can be recovered from
$G_P$ or $K_P$. Indeed, the set $G_P\subset 2^{[m]}$ is partially
ordered by inclusion. This partially ordered set is isomorphic to
the set of faces of $P$ with reversed order, since any face
determines a unique hyperedge of $G_P$. Also $G_P$ is determined
by $K_P$ since every hyperedge in $G_P$ is a multiple intersection
of some maximal simplices in $K_P$. This fact will be verified in
lemma \ref{KPstruct}. Connections of all the structures are
described in section \ref{SecStructKP}.

To get invariants of non-simple polytopes we need definitions for
spaces associated with them. Let us consider a pullback of a
diagram

$$
\xymatrix{
\zp\ar@{->}[d]\ar@{^{(}->}[r]&\Co^m\ar@{->}[d]_p\\
P\ar@{^{(}->}[r]^{j_P}&\Ro_{\geqslant}^m}
$$

The vertical map on the right is given by: $p\colon
(z_1,\ldots,z_m)\mapsto (|z_1|^2,\ldots,|z_m|^2)$.

Like in the case of simple polytopes we call the space $\zp$ a
moment-angle space of a polytope $P$. The space $\zp$ may be
described as an intersection of real quadrics in $\Co^m$ but in
the case of non-simple polytope $P$ this intersection has
singularities. There is a canonical coordinatewise action of a
torus $T^m$ on $\zp$. The orbit space of this action is $P$
itself.

\rem One would expect that the restriction of a moment-angle space
to the facet gives a moment-angle space of this facet multiplied
by complementary torus as it is in the case of simple polytope.
This is not true in general. See example \ref{examplRestrict} in
section \ref{Secmomang}.

For any hypergraph $G$ and a pair of topological spaces $(X,A)$ a
polyhedral product $$(X,A)^G = \bigcup\limits_{\alpha\in
G}X^\alpha\times A^{[m]\setminus \alpha}$$ is defined. This is a
straightforward generalization of a polyhedral product for
simplicial complexes. Indeed, it can be easily proved that
$(X,A)^G = (X,A)^{K_G}$ where $K_G$ is a simplicial closure of
$G$.

As an example of a polyhedral product we would like to consider
moment-angle complex of a hypergraph $\zg = (D^2,S^1)^G$. This
definition coincides with a standard definition when a hypergraph
is a simplicial complex. Also $\zg = \z_{K_G}$ for any hypergraph
$G$.

Another example of a polyhedral product is a complement to the
coordinate space arrangement. Let us fix a hypergraph $G$. For
each hyperedge $e\in G, e\subseteq [m]$ consider the plane $L_e =
\{y\in \Co^m\mid y_i = 0, i\notin e\}$. Then $L = \{L_{e_i}\mid
e_i\in G\}$ is a coordinate space arrangement. Consider the
complement to all spaces from the arrangement $U_G =
\Co^m\setminus\bigcup L_{e_i}$. In fact, we have $U_G =
(\Co,\Co_*)^G$. As in the previous case, $U_G = U_{K_G}$.

\rem The crucial point here is that we can define spaces $U_G$ and
$\zg$ not only for simplicial complexes but also for hypergraphs.
But really we can get nothing new since $U_G = U_{K_G}$ and $\zg =
\z_{K_G}$.

A well known fact states that there is one-to-one correspondence
between simplicial complexes on $m$ vertices and spaces
$U_K\subset \Co^m$ which are complements to coordinate-space
arrangements. In the case of hypergraphs a one-to-one
correspondence between $G$ and $U_G$ does not exist.

\rem Spaces $\zg$ and $U_G$ are homotopically equivalent
(\cite{BP}, section 9).

To summarize all these definitions consider a scheme

\begin{equation}\label{mainscheme}
\xymatrix{
P\ar@{|->}[d]\ar@/_2pc/@{|->}[dd]\ar@{|->}[r]&\zp&\\
G_P\ar@/^1.4pc/@{|->}[rr]\ar@{|->}[d]\ar@{|->}[r]&\zz_{G_P}\ar@{=}[d]\ar@{^{(}->}[r]_{\simeq}&U_{G_P}\ar@{=}[d] \\
K_P\ar@/_1.4pc/@{|->}[rr]\ar@{|->}[r]&\zz_{K_P}\ar@{^{(}->}[r]^{\simeq}&{U_{K_P}}}
\end{equation}
\\
\\
We have described two topological constructions: a moment-angle
space for a polytope and a moment-angle complex of a hypergraph.
The spaces $\zp$ and $\z_{K_P}$ does not coincide in general. We
discuss the connection of these spaces in section \ref{Secmomang}.

Let us review constructions of this section in the case of simple
polytopes.

\begin{prop}
Let $P$ be a simple polytope. Then $G_P=K_P$ and the geometric
realization of $K_P$ coincides with $\partial P^*$. Also $\zp\cong
\z_{K_P}$.
\end{prop}

\begin{proof}
The proof of the first statement easily follows from properties of
simple polytopes (\cite{Zieg}). The last statement is a standard
fact in toric topology \cite{BP}.
\end{proof}

%
%

\section{Interconnections between polytopes and simplicial
complexes}\label{SecConn}
%

Let $P$ be a convex polytope. A scheme shows how simplicial
complex $K_P$ is connected to the polar polytope $P^*$.
$$
\xymatrix{P\ar@{|->}[rr]\ar@{|->}[rd]&&K_P\\
&P^{\ast}\ar@{|->}[ru]&}
$$
Here:
\begin{enumerate}
\item \; $P\mapsto P^{\ast}$ is an association of a polar dual polytope.
\item For a given polytope $Q$ one can build a simplicial complex
$\hat{Q}$. Vertices of $\hat{Q}$ are vertices of $Q$. A set of
vertices forms a simplex of $\hat{Q}$ whenever these vertices are
contained in a common facet of $Q$. In the case when $Q$ is
simplicial there holds $\partial Q = \hat{Q}$.

In the terms introduced above $\widehat{(P^{\ast})} = K_P$. If $P$
is simple, then $K_P = \widehat{(P^{\ast})} = \partial P^{\ast}$.
\end{enumerate}

There is a canonical piecewise linear map from $\hat{Q}$ to
$\partial Q$ (here we do not distinguish between abstract
simplicial complex $\hat{P}$ and its geometrical realization). The
map $p\colon \hat{Q}\to
\partial Q$ is defined on vertices as identity map. Then it can be
continued linearly to each maximal simplex of $\hat{Q}$ thus
giving a piecewise linear map $p\colon \hat{Q}\to \partial Q$.
Such a map is well defined since an image of any maximal simplex
of $\hat{Q}$ is a facet of $Q$.

\begin{prop}\label{propdefretract}
Let $Q$ be a polytope of dimension $n$. There exists a piecewise
linear embedding $i\colon
\partial Q \hookrightarrow \hat{Q}$ such that $p\circ i = id_{\partial Q}$ and
the image $Z = i(\partial Q)$ is a simplicial subcomplex of
$\hat{Q}$ satisfying the properties:

1) $Z$ is homeomorphic to a sphere $S^{n-1}$;

2) $Z$ is a strong deformation retract of $\hat{Q}$.
\end{prop}

\begin{proof}
Let us fix a canonical geometrical realization of a complex
$\hat{Q}$ (that is a realization in Euclidian space such that all
simplices are convex).

Suppose that there is a triangulation of a polytope $Q$ by convex
simplices such that vertices of any simplex from this
triangulation are vertices of $Q$. If $Q$ is triangulated as
described, we consider an induced triangulation $\Delta$ of the
boundary $\partial Q$. Once the triangulation $\Delta$ is fixed,
we define a map $i\colon\partial Q\to\hat{Q}$. The map $i$ is
defined on vertices as identity (vertices of $Q$ are, by
definition, in 1-to-1 correspondence with vertices of $\hat{Q}$).
The map $i$ can be extended by linearity to each simplex of
triangulation $\Delta$. This is well defined map to the complex
$\hat{Q}$ since any simplex $\sigma$ from the triangulation
$\Delta$ is contained in some facet of $Q$ thus its image lies in
some simplex of $\hat{Q}$. Obviously, $i$ is an embedding and $Z =
i(\partial Q)$ is a simplicial subcomplex of $\hat{Q}$. Also, $Z$
is homeomorphic to $\partial Q$ which is homeomorphic to a sphere
$S^{n-1}$. By linearity and all the definitions, we have $p\circ i
= id_{\partial Q}$.

For a point $x\in \partial Q$ the preimage $p^{-1}(x)$ is a convex
subset of $\hat{Q}\subset \Ro^m$. This preimage contains $i(x)$ as
we have already seen. The required strong retraction of $\hat{Q}$
on $i(\partial Q)$ can be easily constructed using linear
contraction of each fiber $p^{-1}(x)$ to a point $i(x)$.

We have proved the proposition under the assumption that any
polytope $Q$ can be triangulated by convex simplices in such a way
that vertices of a triangulation are those of $Q$. To find such a
triangulation we use a technique of Delaunay triangulation on the
set of vertices of $Q$. If $Q$ is a simplex we have nothing to
prove. Otherwise, without loss of generality assume that vertices
of $Q$ do not lie on the same sphere $S^{n-1}$ (if they do, apply
an affine transformation to a polytope $Q$).

We introduce an additional dimension to the linear span of the
polytope so that $Q\subset \Ro^n\subset\Ro^{n+1}$. Consider a
paraboloid in $\Ro^{n+1}$ defined by the equation
$z=x_1^2+\ldots+x_n^2$, where coordinate $z$ corresponds to the
additional dimension. Lift all the vertices of polytope $Q$ to the
paraboloid. More strict: consider a set of points
$A_i=(v_{1i},\ldots,v_{ni}, v_{1i}^2+\ldots+v_{ni}^2)$ for all
vertices $v_i = (v_{i1},\ldots,v_{in})$ of $Q$. Now span a convex
hull $H$ on points $A_i$ in the space $\Ro^{n+1}$ and take its
lower part $L$ (this is a standard technique when dealing with
Delaunay triangulations, see \cite{Gal} for information). The
lower part $L$ is subdivided to convex polytopes. If we take a
projection of $L$ back to the subspace $\Ro^n\subset\Ro^{n+1}$ we
get a subdivision of $Q$ by convex polytopes $R_j$; each of them
is spanned by some vertices of $Q$. Moreover, all $R_j$ have less
vertices than $Q$. Indeed, if $Q$ is the only lower face of
$H\subset \Ro^{n+1}$, then all the points $A_i$ lie in the same
hyperplane. Therefore, all vertices of the polytope $Q$ lie on a
sphere (see \cite{Gal}) which contradicts the assumption.

We can construct a triangulation of each $R_j$ by induction on the
number of vertices. The proof is now completed.
\end{proof}

The construction of $\hat{Q}$ can be viewed as simplicial
resolution of a polytope $Q$. We suggest that it can be used in
studying singular vertices of nonsimple polytopes.

\begin{prop}\label{crossjoin}
Let $P$ and $R$ be convex polytopes. Then $K_{P\times R} = K_P\ast
K_R$ and $\widehat{P\circ R} = \hat{P}\ast\hat{R}$, where $\times$
is a direct product of polytopes, $\circ$ is a convex product of
polytopes and $\ast$ is a join of simplicial complexes (see
\cite{BP} for definitions of these operations).
\end{prop}

\begin{proof}
Denote facets of $P$ by $\F_1,\ldots,\F_m$ and facets of $R$ by
$\F_1',\ldots,\F_l'$. Then facets of $P\times R$ are $\F_1\times
R,\ldots,\F_m\times R,P\times\F_1',\ldots,P\times\F_l'$. Then
$\F_{i_1}\times R\cap\ldots\cap\F_{i_k}\times R\cap
P\times\F_{j_1}'\cap\ldots\cap P\times\F_{j_l'}'\neq\varnothing$
if and only if $\F_{i_1}\cap\ldots\cap\F_{i_k}\neq\varnothing$ and
$\F_{j_1}'\cap\ldots\cap\F_{j_l'}'\neq\varnothing$. This means
that $K_{P\times R}=K_P\ast K_R$.

Second fact follows from the series of identities $\widehat{P\circ
R} = \widehat{(P^*\times R^*)^*} = K_{P^*\times R^*}=K_{P^*}\ast
K_{R^*}=\hat{P}\ast\hat{R}$.
\end{proof}


\rem Combinatorial invariants of simplicial complexes $K_P$ and
$\widehat{P}$ are the invariants of a polytope $P$. For example,
face vector of $K_P$ is an invariant of $P$ (see section
\ref{SecStructKP}). In sections \ref{Secbuch} and \ref{Secmomang}
Buchstaber number and Betti numbers of $K_P$ are considered.

%

\section{The structure of $K_P$}\label{SecStructKP}

The key question for us is the following: given an arbitrary
simplicial complex $K$, is there a convex polytope $P$ such that
$K=K_P$?

The complete answer to this question seems to be very difficult
and unattainable (\cite{Barn},\cite{BP}). Indeed, the problem
being restricted to the class of simple polytopes has the form:
which simplicial complexes arise as boundaries of simplicial
polytopes? Any complex of that kind is a simplicial sphere, but
not all the spheres are boundaries of simplicial polytopes. The
Barnette sphere \cite{Barn} is not a boundary of any convex
simplicial polytope. See other examples and references in
\cite{BP}.

Here we discuss the general situation: what are the necessary
conditions for a simplicial complex to be equivalent to a complex
$K_P$ for some convex polytope $P$. We found such a condition and
as a byproduct get the formula for the $f$-vector of a simplicial
complex $K_P$.

We will need a small technical lemma to work with $K_P$.

\begin{lemma}\label{lemtwocovers}
Let $\mathcal{C}_2\subseteq\mathcal{C}_1$ be two finite
contractible covers of the same space. Then nerve $\ncal_2$ of the
cover $\ccal_2$ is a strong deformation retract of the nerve
$\ncal_1$ of the cover $\ccal_1$.
\end{lemma}

\begin{proof}
Without loss of generality it can be assumed that the cover
$\ccal_2=\{A_1,\ldots,A_k\}$ and $\ccal_1 =
\ccal_2\cup\{A_{k+1}\}$, that is $\ccal_1$ has only one extra
element which is not contained in $\ccal_2$. Then, obviously,
$\ncal_2\subseteq\ncal_1$. Moreover, we may write $\ncal_1 =
\ncal_2\cup_{\link v_{k+1}}\sta v_{k+1}$, where $v_{k+1}$ is a
vertex of $\ncal_1$ which corresponds to a set $A_{k+1}$, and
$\link$ and $\sta$ are taken in the complex $\ncal_1$. By all
definitions, $\link_{\ncal_1}v_{k+1}$ is a nerve of a cover of
$A_{k+1}$ by the sets $B_i = A_i\cap A_{k+1}$. Since all the sets
and their intersections are contractible we found that
$\link_{\ncal_1}v_{k+1}$ is contractible. But $\ncal_1 =
\ncal_2\cup_{\link v_{k+1}}\sta v_{k+1} = \ncal_2\cup_{\link
v_{k+1}}\cone(\link v_{k+1}) = \ncal_2\cup_Y\cone(Y)$, where $Y$
is contractible. By standard reasoning, $\ncal_2$ is a strong
deformation retract of $\ncal_1$.
\end{proof}

Let $P$ be a convex polytope and $\F_i$ --- its facets. Denote the
poset of all faces of a polytope $P$ by $\mathcal{P}$. We now
define a map $\tilde{\sigma}$ from $\mathcal{P}$ to $K_P$ by the
formula $\tilde{\sigma}(F) = \{i\mid F\subseteq\F_i\}$. This map
was already defined in section \ref{Secdefinit} (but we considered
it as a map defined on the points of $P$). The map
$\tilde{\sigma}$ is injective and its image is exactly the
hypergraph $G_P$ by definition. Simplices from the image of
$\tilde{\sigma}$ have very important property.

\begin{prop}\label{propdefrtactPoly}
Let $F$ be a face of a polytope $P$. Consider a simplex
$\tilde{\sigma}(F)$ and its link
$L_{\tilde{\sigma}(F)}=\link\tilde{\sigma}(F)$ in a simplicial
complex $K_P$. In this case there is a simplicial subcomplex
$Z_{\tilde{\sigma}(F)}\subseteq L_{\tilde{\sigma}(F)}$ such that:

1) $Z_{\tilde{\sigma}(F)}$ is a strong deformation retract of
$L_{\tilde{\sigma}(F)}$;

2) $Z_{\tilde{\sigma}(F)}$ is homeomorphic to a sphere $S^{\dim
F-1}$.
\end{prop}

\begin{proof}
First of all, we prove that $L_{\tilde{\sigma}(F)}$ is
homotopically equivalent to a sphere $S^{\dim F-1}$.

We have $F = \bigcap_{i\in \tilde{\sigma}(F)}\F_i$ and, moreover,
$F\nsubseteq \F_j$ if $j\notin \tilde{\sigma}(F)$ by definition of
function $\tilde{\sigma}$. The boundary $\partial F$ is covered by
the sets $G_j=\F_j\cap F = \F_j\cap\left(\bigcap_{i\in
\tilde{\sigma}(F)} \F_i\right)$ because any point of the boundary
$\partial F$ is contained in some facet which does not contain $F$
(these sets does not cover $F$ itself by the reasoning above). We
claim that the nerve of this cover is exactly the complex $\link
\tilde{\sigma}(F)$. To prove this consider the series of
equivalent statements: $G_{j_1}\cap\ldots\cap G_{j_s}\neq
\varnothing \Leftrightarrow
\F_{j_1}\cap\dots\cap\F_{j_s}\cap\left(\bigcap_{i\in
\tilde{\sigma}(F)}\F_i\right) \Leftrightarrow
\{j_1,\ldots,j_s\}\sqcup \tilde{\sigma}(F)\in K_P \Leftrightarrow
\{j_1,\ldots,j_s\}\in L_{\tilde{\sigma}(F)}$. All sets $G_j$ are
convex so the cover of $\partial F$ by $G_i$ is contractible.
Therefore $L_{\tilde{\sigma}(F)}\simeq \partial F \simeq S^{\dim F
- 1}$.

Now we construct a strong deformation retract
$Z_{\tilde{\sigma}(F)}$ of $L_{\tilde{\sigma}(F)}$ which is
homeomorphic to a sphere $S^{\dim F - 1}$. Note that a face $F$ is
itself a polytope. Thus the complex $K_F$ is defined. We will
follow the plan: 1) prove that $K_F$ is a strong deformation
retract of $L_{\tilde{\sigma}(F)}$ 2) prove that there is a
simplicial subcomplex $Z_{\tilde{\sigma}(F)}\subseteq K_F$ which
is a strong deformation retract of $K_F$ and which is homeomorphic
to a sphere $S^{\dim F-1}$.

1) There are two contractible covers of a set $\partial F$. First
cover $\mathcal{A}=\{G_j\}$ was defined earlier in the proof.
Second cover is the cover $\mathcal{B}=\{H_t\}$ of the boundary
$\partial F$ by the facets of $F$. The nerve of $\mathcal{A}$ is a
complex $\link \tilde{\sigma}(F)$ as was previously proved. The
nerve of $\mathcal{B}$ is $K_F$ by definition. It can be easily
seen that each facet $H_t$ of the cover $\mathcal{B}$ is contained
in the cover $\mathcal{A}$. Therefore $\mathcal{B}$ is a subcover
of $\mathcal{A}$.

The lemma \ref{lemtwocovers} being applied to covers
$\bcal\subseteq\acal$ states that $K_F$ is a strong deformation
retract of $\link \tilde{\sigma}(F)$. It is obviously a simplicial
subcomplex.

2) We should now prove that $K_F$ has a subcomplex
$Z_{\tilde{\sigma}(F)}$ which is a strong deformation retract of
$K_F$ and is homeomorphic to a sphere. By proposition
\ref{propdefretract}, the boundary of the dual polytope $\partial
F^*$ can be embedded in $K_F$ as a strong deformation retract.
\end{proof}

The proposition \ref{propdefrtactPoly} states that some of the
simplices of $K_P$ have very specific links. This condition is
axiomatized in the following construction.

Let $K$ be an arbitrary simplicial complex. Its maximal under
inclusion simplices form a set $M(K)\subset K$. We call a simplex
$\tau$ face simplex if it can be represented as an intersection of
some number of maximal simplices $\tau =
\sigma_1\cap\ldots\cap\sigma_s$, $\sigma_i\in M(K)$. The set of
all face simplices will be denoted by $F(K)$. Therefore
$M(K)\subseteq F(K)\subseteq K$. The set $F(K)$ is naturally
ordered by inclusion.

\begin{defin}\label{definScomplex}
Simplicial complex $K$ is called polytopic complex (P-complex) of
rank $n$ if the following conditions hold:

1) $\varnothing\in F(K)$, i.e. intersection of all maximal
simplices of $K$ is empty;

2) $F(K)$ is a graded poset of rank $n$ (it means that all its
maximal under inclusion chains have the same length $n+1$). In
this case the rank function $\rk(\sigma)$ is defined. It satisfies
the property $\rk(\varnothing) = 0$, and $\rk(\sigma) = n$ for any
maximal simplex $\sigma$;

3) Suppose $\sigma\in F(K)$. Then there is a simplicial subcomplex
$Z_{\sigma}$ in simplicial complex $\link_K \sigma$ such that
$Z_{\sigma}$ is homeomorphic to a sphere $S^{n-\rk(\sigma)-1}$ and
$Z_{\sigma}$ is a strong deformation retract of $\link_K \sigma$.
Here, by definition, $\link\varnothing=K$ and
$S^{-1}=\varnothing$.

The $P$-complex $K$ is called reduced if any singleton is a face
simplex: $\{v\}\in F(K)$.
\end{defin}

\rem The rank function on a poset has the property $\rk(\sigma) =
\rk(\tau)-1$ if $\sigma\subset\tau$ and there is no element $\rho$
such that $\sigma\subset\rho\subset\tau$. Consult \cite{St} for
details.

\rem It follows from definition that in a $P$-complex link of any
face simplex is homotopically equivalent to a sphere.

\begin{figure}[h]
\begin{center}
\includegraphics[scale=0.3]{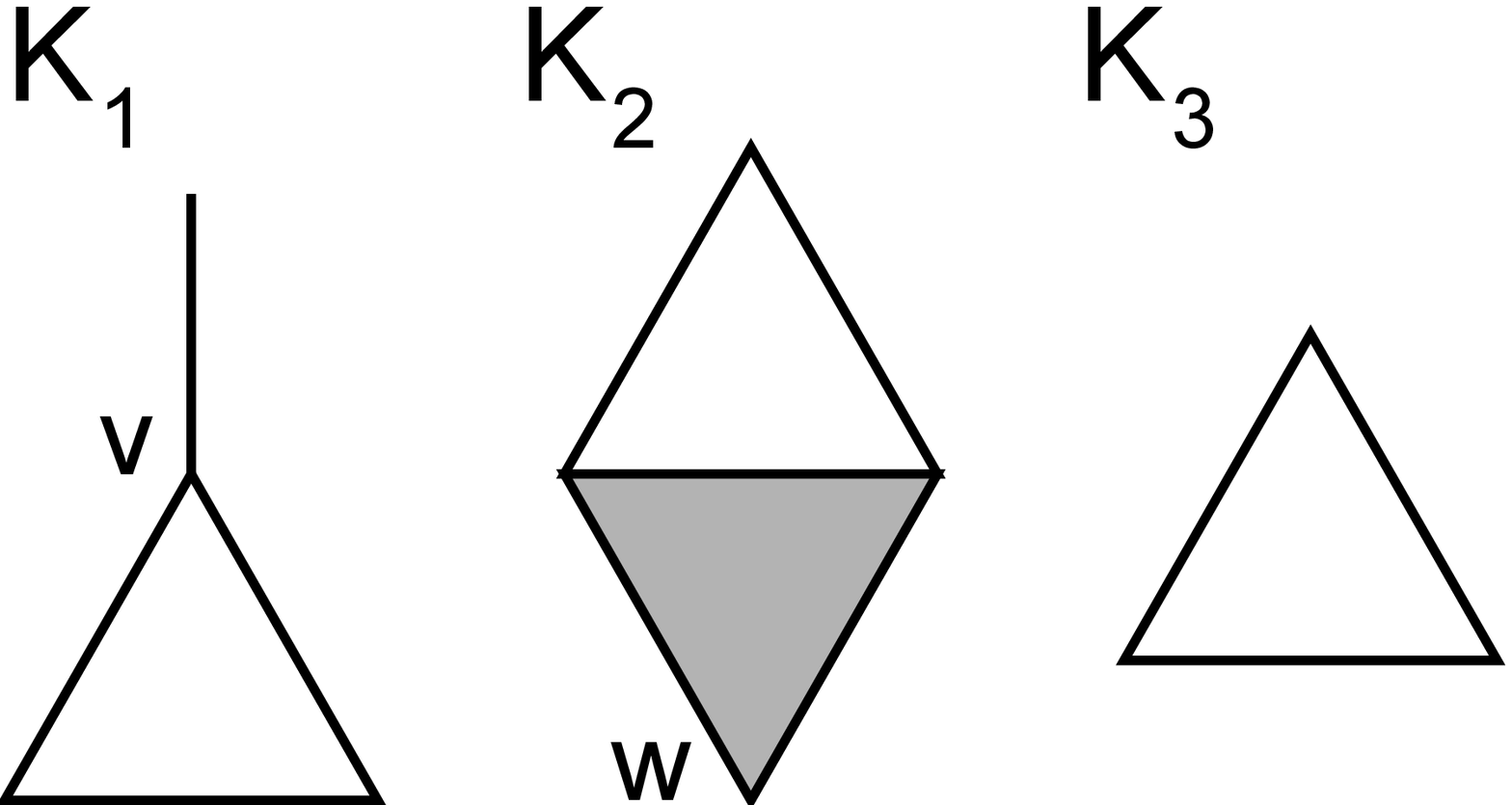}
\end{center}
\caption{}\label{FigScomp}
\end{figure}

\ex Figure \ref{FigScomp} illustrates three simplicial complexes.
The complex $K_1$ is not a $P$-complex. Indeed, $v$ is a face
simplex (it can be represented as an intersection of maximal
simplices), but $\link v$ is a disjoint union of three points.
Therefore $\link v$ is not homotopically equivalent to a sphere of
any dimension. This violates the third condition in definition
\ref{definScomplex}.

The complex $K_2$ is $P$-complex of rank 2. It is unreduced since
the vertex $\{w\}$ is not an intersection of maximal simplices.
The complex $K_3$ is an example of reduced $P$-complex of rank 2.

\begin{lemma}\label{KPstruct}
For any $n$-dimensional polytope $P$ with $m$ facets the complex
$K_P$ is a reduced $P$-complex of rank $n$. Moreover, $F(K_P)=G_P$
as a subset of $2^{[m]}$ and $F(K_P)$ is equivalent as a graded
poset to the set of faces of $P$ ordered by reverse inclusion.
\end{lemma}

\begin{proof}
At first, we will prove that $F(K_P)$ is equal to $G_P =
\tilde{\sigma}(\mathcal{P})$.

It can be seen that vertices of $P$ correspond to maximal
simplices of $K_P$, since they are contained in the most number of
facets. Now recall the basic fact: the face $F$ is contained in
the facet $\F$ iff all vertices of $F$ are contained in $\F$.
Therefore $\{i\mid F\subseteq \F_i\} = \bigcup_{v\in F}\{i\mid
v\subseteq\F_i\}$. This expression shows that every element from
the image of $\tilde{\sigma}$ is an intersection of some maximal
simplices. Therefore $G_P = \tilde{\sigma}(\mathcal{P})\subseteq
F(K_P)$.

To prove that $\tilde{\sigma}$ is surjective map to $F(K_P)$
consider an arbitrary face simplex $\sigma$ of $K_P$. By
definition, $\sigma = \tau_1\cup\ldots\cup\tau_s$, where each
$\tau_j$ is maximal. Each $\tau_j$ have the form $\tau_j=\{i\mid
v_j\subseteq \F_i\}$ for some vertex $v_j$ of a polytope $P$.
Summarizing, we have $\sigma = \{i\mid \F_i$ contains all $v_j$
for $j=1,\ldots,s\}$. Let $F$ be the minimal face of $P$
containing all vertices $v_j$. Then, obviously $\tilde{\sigma}(F)
= \sigma$.

So forth there is a bijection between sets $\mathcal{P}$, $G_P$
and $F(K_P)$ which respects the order. So the poset $F(K_P)$ is an
ordered poset and one can define the rank function on it by
$\rk(\sigma) = n-\dim(\tilde{\sigma}^{-1}(\sigma))$. We have
$\rk(\varnothing) = 0$ since $\tilde{\sigma}^{-1}(\varnothing) =
P$ and $\dim P = n$. Also $\rk(\sigma)=n$ when $\sigma$ is maximal
simplex since $\tilde{\sigma}^{-1}(\sigma)$ is vertex in this
case.

It is clear that any vertex of $K_P$ is a face simplex. Indeed,
each vertex $i$ of $K_P$ corresponds to the facet $\F_i$ thus
$\tilde{\sigma}(\F_i) = \{i\}\subseteq F(K_P)$. The empty set is
an intersection of all vertices, so the first condition also
holds.

The last thing we need to prove is property 3 in definition
\ref{definScomplex} which concerns links of face simplices. But
this is exactly the statement of proposition
\ref{propdefrtactPoly}. This completes the proof of the theorem.
\end{proof}

\begin{figure}[h]
\begin{center}
\includegraphics[scale=0.3]{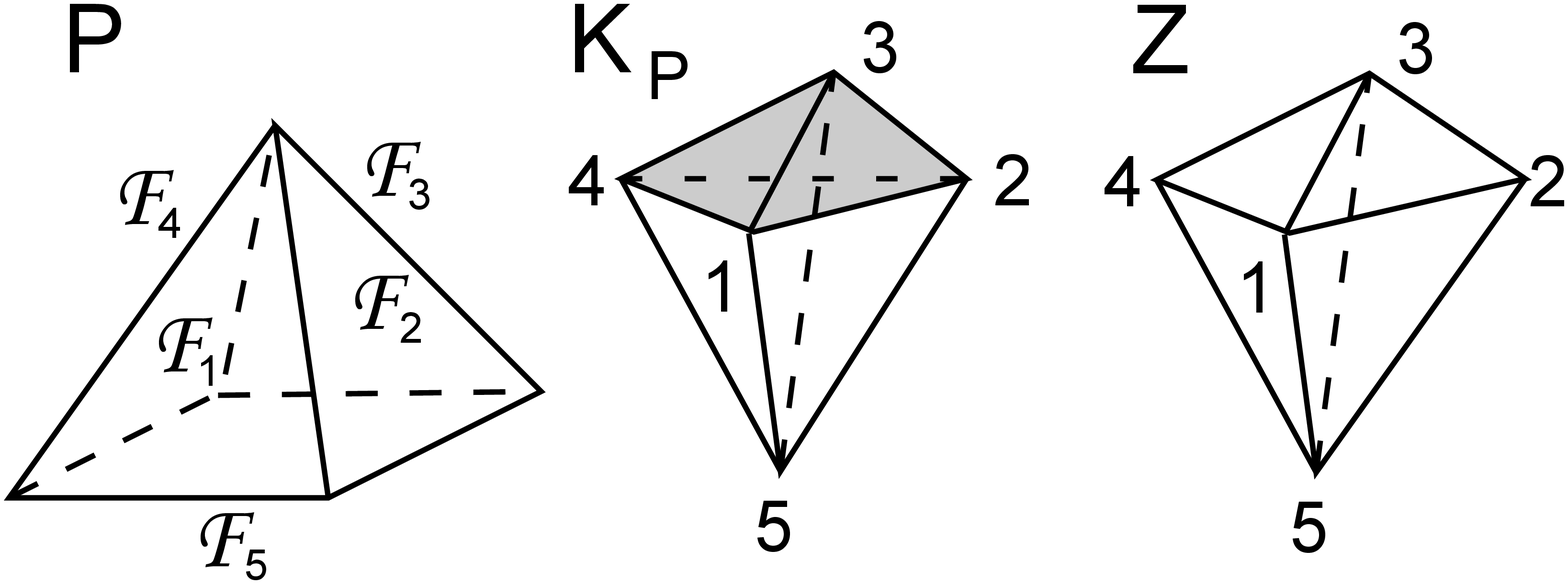}
\end{center}
\caption{}\label{FIGpyrKP}
\end{figure}

\ex\label{examplSqPyr5} Consider a pyramid $P$ over a square as in
example \ref{examplPyrSq}. Figure \ref{FIGpyrKP} shows the
polytope $P$, simplicial complex $K_P$ and the subcomplex
$Z\subset K_P$. The complex $Z = Z_{\varnothing} =
Z_{\tilde{\sigma}(P)}$ is composed of maximal simplices
$\{1,3,4\}, \{1,4,5\}, \{3,4,5\}, \{1,2,3\}, \{1,2,5\},
\{2,3,5\}$. All of them are the simplices of $K_P$. The complex
$Z$ is homeomorphic to the sphere $S^2$. It is a retract of $K_P$
as well.

\begin{figure}[h]
\begin{center}
\includegraphics[scale=0.2]{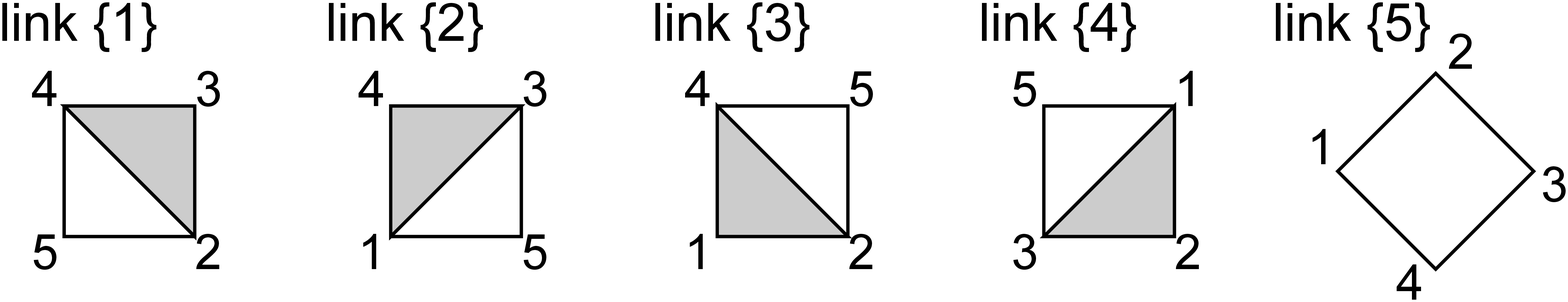}
\end{center}
\caption{}\label{FIGpyrLinks}
\end{figure}

Now we check that $K_P$ is a $P$-complex. This is the list of all
face simplices of $K_P$: maximal simplices $\{1,2,3,4\},
\{1,2,5\}, \{2,3,5\}, \{3,4,5\}, \{1,4,5\}$; double intersections
of maximal simplices $\{1,2\}, \{2,3\}, \{3,4\}, \{1,4\}, \{1,5\},
\{2,5\}, \{3,5\}, \{4,5\}$ (all of them correspond to edges of $P$
as they should); the vertices $\{1\}, \{2\}, \{3\}, \{4\}, \{5\}$
and the empty set $\varnothing$. Links of the vertices are
depicted in figure \ref{FIGpyrLinks}. Note that all of them
contain a circle $S^1$ as a retract. We do not illustrate links of
all other face simplices. But one can easily check the property 3
from the definition of $P$-complex. For example, $\link \{1,2\}$
is a disjoint union of a point $\{5\}$ and an interval $\{3,4\}$,
thus contains $S^0$ as a retract.

But if we look at the simplex $\{1,3\}$ which is not a face
simplex, we see that $\link \{1,3\}$ is an interval, that is a
contractible space. The situation goes the same with all other
nonface simplices.

\ex\label{examplSphere} Any $(n-1)$-dimensional $PL$-sphere $K$ is
an $P$-complex of rank $n$ and all its simplices are face
simplices, so $F(K) = K$. Indeed, any $(n-2)$-dimensional simplex
of $K$ is contained in exactly two maximal simplices thus it is a
face simplex. Suppose $\sigma\in K$ is an arbitrary simplex
contained in some maximal simplex $\tau$. Then $\sigma$ can be
represented as an intersection of $n-2$-dimensional subsimplices
of $\tau$, thus $\sigma$ is a face simplex. The rank function is
given by $\rk\sigma = \dim\sigma+1$. The third condition of
definition \ref{definScomplex} holds by definition of PL-sphere.

\ex\label{exampl1Scomplex} Any $P$-complex $K$ of rank 1 is a
disjoint union of two simplices. Indeed, its maximal simplices
should not intersect by definition. On the other hand, $K$ should
be homotopically equivalent to a pair of points.

\begin{prop}\label{propPlinkP} If $K$ is a $P$-complex of rank $n$ and $\sigma$ its face
simplex of rank $k$, then $\link \sigma$ is a $P$-complex of rank
$n-k$.
\end{prop}

\begin{proof} To prove this, consider a maximal simplex $\tau$ of the
complex $\link \sigma$. Then $\sigma\sqcup\tau$ is the maximal
simplex of $K$. And vice a versa: all maximal simplices of $K$
containing $\sigma$ have the form $\sigma\sqcup\tau$ where $\tau$
is the maximal simplex of $\link \sigma$. So $\tau\in
M(\link\sigma)\Leftrightarrow \tau\sqcup\sigma\in M(K)$. The same
obviously holds for face simplices: $\tau\in
F(\link\sigma)\Leftrightarrow \tau\sqcup\sigma\in F(K)$. Therefore
$F(\link\sigma)$ is isomorphic as a poset to the filter
$F(K)_{\geqslant} = \{\tau\in F(K)\mid \tau\supseteq\sigma\}$. So
the poset $F(\link\sigma)$ is a graded poset of rank $n-k$ and the
rank function is defined by $\rk_{\link\sigma}\tau
=\rk_K(\tau\sqcup\sigma)-k$ (there stands rank function on
$F(K)$at the right). The third condition in definition
\ref{definScomplex} follows from identities:
$\link_{\link\sigma}\tau = \link(\sigma\sqcup\tau)\simeq
S^{n-\rk_{K}(\sigma\sqcup\tau)-1}=S^{n-k-\rk_{\link\sigma}\tau-1}$.
\end{proof}

\ex The only reduced $P$-complexes of rank $2$ are the boundaries
of polygons. Let $K$ be a reduced $P$-complex of rank $2$.
Consider two sets: $M(K)$ --- the set of maximal simplices of $K$
and $V(K)$ --- the set of vertices. Each vertex is a face simplex
and has rank 1. Thus the link of any vertex is a disjoint union of
two simplices by proposition \ref{propPlinkP} and example
\ref{exampl1Scomplex}. Therefore any vertex is contained in
exactly two maximal simplices. Consider a graph $\Gamma$ on a set
$M(K)$: two nods $\sigma_1$ and $\sigma_2$ are connected by an
edge iff corresponding maximal simplices $\sigma_1$ and $\sigma_2$
share a common vertex. This graph is homotopically equivalent to
$K$. One can prove this by considering the contractible cover of
geometrical realization of $K$ by maximal simplices. So forth,
$\Gamma$ is homotopically equivalent to a 1-sphere. Therefore it
has the same number of edges and nodes. This means that $\#V(K) =
\#M(K)$. Each maximal simplex has at least two vertices. Counting
the number of pairs $v\subset m$, where $v\in V(K)$ and $m\in
M(K)$, we find: $2\#V(K) = \#\{v\subset m\}\geqslant 2\#M(K)$.
Since $\#V(K) = \#M(K)$ we find out that any maximal simplex have
exactly two vertices, providing $K$ to be a graph itself. Since
the degree of any node of $K$ is 2 and $K$ is homotopically
equivalent to a circle we conclude that $K$ is a simple cycle,
i.e. a boundary of a polygon.
\\
\\
Let $\tau$ be a simplex of $K$. If $\sigma_1,\sigma_2\in F(K)$ and
$\tau\in\sigma_1$, $\tau\in\sigma_2$, then
$\sigma_1\cap\sigma_2\in F(K)$ and $\tau\in \sigma_1\cap\sigma_2$.
Therefore for every $\tau\in K$ there exists a minimal $\sigma\in
F(K)$ containing $\tau$ (it is an intersection of all face
simplices containing $\tau$). We denote such face simplex by
$\hat{\tau}$.

\begin{lemma}
For any $\sigma\in K$ we have $\link \sigma = \link
\hat{\sigma}\ast (\hat{\sigma}\setminus\sigma)$.
\end{lemma}

\begin{proof}
Any maximal simplex of $K$ containing $\sigma$ contains
$\hat{\sigma}$ as well. Let $\tau$ be the maximal simplex of
$\link\sigma$. Then $\sigma\sqcup\tau$ is a maximal simplex of $K$
thus $\hat{\sigma}\setminus\sigma\subseteq\tau$. Therefore any
maximal simplex $\tau$ of $\link\sigma$ can be written as $\tau =
(\hat{\sigma}\setminus\sigma)\sqcup\check{\tau}$, where
$\check{\tau}$ is the maximal simplex of $\link\hat{\sigma}$. This
yields $\link \sigma = \link \hat{\sigma}\ast
(\hat{\sigma}\setminus\sigma)$.
\end{proof}

\begin{cor}
If $\sigma$ is not a face simplex in complex $K$, then
$\link\sigma$ is contractible.
\end{cor}

We see that in a $P$-complex links of all simplices are either
contractible or homotopically equivalent to a sphere. To make use
of this fact we observe some constructions which were used in the
theory of the ring of polytopes \cite{Buch}. This combinatorial
technique was introduced to effectively count the number of faces
of simple polytopes, in particular, to prove the Dehn-Sommerville
relations. The same technique works in our situation as well.

Let $K$ be an arbitrary simplicial complex. Consider an
$f$-polynomial
\begin{equation}\label{definFsimpcompl}
f_K(t) = \sum\limits_{\sigma\in K}t^{|\sigma|} =
f_0+f_1t+f_2t^2+\ldots+f_nt^n+\ldots,
\end{equation}
where $f_i$ is the number of $(i-1)$-dimensional simplices of $K$,
$f_0=1$.

\begin{lemma}\label{difoff}
For any simplicial complex $K$ we have
\begin{equation}\label{equatDiffSC}
\frac{d}{dt}f_K(t) = \sum\limits_{v\in V(K)}f_{\link v}(t).
\end{equation}
\end{lemma}

\begin{proof}
The proof consists in the computation:
\begin{multline*}
\sum\limits_{v\in K} f_{\link v}(t) = \sum\limits_{v \in K}
\sum\limits_{\sigma \in \link v} t^{|\sigma|} = \\ =
\sum\limits_{\substack{v \in K, \sigma \in K,\\ v \notin \sigma, v
\cup \sigma \in K}} t^{|\sigma|} = \sum\limits_{\substack{\tau \in
K, v \in \tau,\\ \sigma = \tau \backslash v}}t^{|\sigma|} =
\sum\limits_{\tau \in K, v \in \tau} t^{|\tau| - 1} = \\ =
\sum\limits_{\tau \in K} |\tau| t^{|\tau|-1} =
\frac{\partial}{\partial t}f_K(t).
\end{multline*}
\end{proof}

\begin{lemma}\label{difs}
For any simplicial complex $K$ and for any natural $s$ we have
$$\left(\frac{d}{dt}\right)^sf_K(t) =
s!\sum\limits_{\sigma\in K, |\sigma|=s} f_{\link\sigma}(t).$$
\end{lemma}

\begin{proof}
The proof uses an induction on $s$. The base of induction $s=1$
was proved in lemma \ref{difoff}. By induction hypothesis
$\left(\frac{d}{dt}\right)^{s-1}f_K(t) =
(s-1)!\sum\limits_{\sigma\in K, |\sigma|=s-1} f_{\link\sigma}(t)$.
Differentiating this equality by $t$ and using lemma \ref{difoff},
we get
\begin{multline*}
\left(\frac{d}{dt}\right)^sf_K(t) = (s-1)!\sum\limits_{\sigma\in
K,
|\sigma|=s-1}\sum\limits_{v\in\link\sigma}f_{\link_{\link\sigma}v}(t)
=\\= (s-1)!\sum\limits_{\sigma\in K,
|\sigma|=s-1}\sum\limits_{v\in\link\sigma} f_{\link(\sigma\sqcup
v)}(t) = (s-1)!\sum\limits_{\tau\in K,
|\tau|=s}\sum\limits_{v\in\tau}f_{\link\tau}(t)=\\=s!\sum\limits_{\tau\in
K, |\tau|=s}f_{\link\tau}(t),
\end{multline*}
which was to be proved.
\end{proof}

\begin{theorem}\label{fofK}
For any $P$-complex $K$ of rank $n$ we have
\begin{equation}\label{equatF}
f_K(t) = \sum\limits_{\sigma\in
F(K)}(-1)^{n-\rk\sigma}(t+1)^{|\sigma|}.
\end{equation}
\end{theorem}

\begin{proof}
First of all, note that for any simplicial complex $L$ there holds
$f_L(-1) = 1 - \chi(L)$, where $\chi(L)$ is an Euler
characteristic of a simplicial complex. In particular,
$$f_L(-1) = \begin{cases}0, \mbox{  if }L\simeq \pt;
\\(-1)^{l+1},\mbox{  if } L\simeq S^l.
\end{cases}$$
Now consider a Taylor series for the polynomial $f_K(t)$ at a
point $-1$:
\begin{equation}\label{taylor}
f_K(t) = f_K(-1)+\frac{1}{1!}\frac{df_K}{dt}(-1)(t+1)+
\ldots+\frac{1}{s!}\frac{d^sf_K(-1)}{dt^s}(t+1)^s+\ldots.
\end{equation}
Using lemma \ref{difs} and the observation about Euler
characteristic, we have $$\frac{1}{s!}\frac{d^sf_K(-1)}{dt^s} =
\sum\limits_{\sigma\in K, |\sigma|=s} f_{\link\sigma}(-1) =
\sum\limits_{\sigma\in F(K), |\sigma|=s} (-1)^{n-\rk\sigma}.$$ So
the formula \ref{taylor} may be extended:
$$f_K(t) = \sum\limits_{s=0}^\infty\left[\sum\limits_{\sigma\in F(K),
|\sigma|=s}(-1)^{n-\rk\sigma}\right](t+1)^s=\sum\limits_{\sigma\in
F(K)}(-1)^{n-\rk\sigma}(t+1)^{|\sigma|},$$ which gives the
required result.
\end{proof}

Let $P$ be a convex polytope and $F$ --- its face. Let $m(F)$ be a
number of facets containing $F$.

\begin{defin}
Two-dimensional face polynomial of a polytope $P$ is the
polynomial in two variables
\begin{equation}\label{definFmy}
F_P(\alpha, t) = \sum\limits_{F\subseteq P} \alpha^{\dim
F}t^{m(F)},
\end{equation}
where the sum is taken over all faces of $P$ including $P$ itself.
It is clear that two-dimensional polynomial is a combinatorial
invariant of $P$
\end{defin}

From theorem \ref{fofK} and lemma \ref{KPstruct} follows

\begin{cor}\label{fofKP}
For any convex polytope $P$ there holds $f_{K_P}(t) =
F_P(-1,t+1).$
\end{cor}

We now observe briefly some basic facts from the theory of ring of
polytopes developed in \cite{Buch}. We would like to get some
results in this theory as a particular cases of general
considerations made above.

Let $\mathbb{P}_n$ be a free abelian group formally generated by
combinatorial simple polytopes of dimension $n$ (that is a group
of all finite formal sums $\alpha_1P_1+\ldots+\alpha_sP_s$, where
$\alpha_i\in\Zo$ and $P_s$ are simple combinatorial polytopes).
Taking direct sum over $n$, we get a group of all simple
polytopes: $\mathbb{P}=\bigoplus_n\mathbb{P}_n$. This abelian
group possesses the structure of a graded differential ring. The
multiplication is defined on basis elements by the formula $P\cdot
Q = P\times Q$ and extended by linearity on the whole group. The
differentiation is defined on basis elements by $dP =
\sum_{\F\subset P} \F$, which associates to a simple polytope a
formal sum of its facets. Differentiation is extended by linearity
to the whole group $\mathbb{P}$.

To each simple polytope $P\in \mathbb{P}_n$ corresponds a face
polynomial $F(P)$ in two variables
\begin{equation}\label{definBuchF}
F(P)(\alpha,t) = \sum_{F\subseteq P}\alpha^{\dim F}t^{n-\dim F},
\end{equation}
where sum is taken over all faces including $P$ itself. This
correspondence can be extended by linearity to the linear map
$F\colon \mathbb{P}\to \Zo[\alpha,t]$ to the ring of polynomials.
Then $F$ is a differential ring homomorphism, in particular, the
commutation formula holds:
\begin{equation}\label{commutformula}
F(dP) = \frac{\partial}{\partial t}F(P)
\end{equation}

This result is used to give a simple proof of Dehn-Sommerville
relations by induction on $n$. The Dehn-Sommerville relations can
be written in the following form: for any simple polytope $P$
there holds
\begin{equation}\label{dehnsommclassic}
F(P)(1,t) = F(P)(-1,t+1).
\end{equation}

The next statement shows how the general theory of $K_P$ works for
simple polytope $P$.


\begin{prop}\label{linktoring}
Let $P$ be a simple polytope of dimension $n$.

1)The polynomial $F_P(\alpha, t)$ coincides with a face polynomial
$F(P)(\alpha,t)$ defined by \ref{definBuchF}.

2)Formula \ref{commutformula} follows from lemma \ref{difoff}
applied to the complex $K_P$.

3)Formula \ref{dehnsommclassic} follows from formula \ref{fofKP}
applied to the complex $K_P$.
\end{prop}

\begin{proof}
We give a sketch of a proof. Technical details can be easily
restored.

1)For a face $F$ of a simple polytope $P$ we have $m(F) = n-\dim
F$ (that is a classical fact). Therefore $F_P(\alpha, t)$ is a
homogeneous polynomial and coincides with the face polynomial
$F(P)(\alpha,t)$.

2)Let $\F_i$ be a facet of $P$ and $i$ be a vertex of $K_P$
corresponding to this facet. If $P$ is simple, then $K_{\F_i} =
\link_{K_P}i$.

The complex $K_P$ is $(n-1)$-sphere in the case of simple polytope
and by example \ref{examplSphere} any simplex of $K_P$ is a face
simplex. Therefore $f_{K_P}(t) = \sum\limits_{\sigma\in
K_P}t^{|\sigma|} = \sum\limits_{\sigma\in F(K_P)}t^{|\sigma|} =
F_P(1,t) = F(P)(1,t)$.

If we forget about complementary variable $\alpha$ in formula
\ref{commutformula} and use the definitions we get
$\sum_{\F_i\subset P}F(\F_i)(1,t) = \frac{d}{dt}F(P)(1,t)$.
Substituting $F(\F)(1,t)$ by $f_{\link i}(t)$ and $F(P)(1,t)$ by
$f_{K_P}(t)$, formula \ref{commutformula} reduces exactly to
\ref{difoff}.

3)To prove the third part of the statement note that $f_{K_P}(t) =
F_P(-1,t+1)$ by the corollary \ref{fofKP}. On the other hand,
$f_{K_P} = F_P(1,t)$ by the observation above. This gives
Dehn-Sommerville relations $F(P)(1,t) = F(P)(-1,t+1)$.
\end{proof}

Proposition \ref{linktoring} serves as a link between theory of
ring of polytopes and the general construction of $K_P$.
\\
\\
We now turn to the general situation when $P$ is nonsimple.

\begin{lemma}\label{lemIneqEq}
Let $P$ be a convex polytope. Then there is a coefficient-wise
inequality $F_P(1,t)\leqslant F_P(-1,t+1)$, which turns to an
equality for terms $t^0$ and $t^1$.
\end{lemma}

\begin{proof}
By definition, $F_P(1,t)=\sum\limits_{\sigma\in
F(K_P)}t^{|\sigma|}$. By corollary \ref{fofKP}, we have
$F_P(-1,t+1) = \sum\limits_{\sigma\in K_P}t^{|\sigma|}$. But
$\sum\limits_{\sigma\in F(K_P)}t^{|\sigma|}\leqslant
\sum\limits_{\sigma\in K_P}t^{|\sigma|},$ since the second sum is
taken over greater set.

Complex $K_P$ is a reduced $P$-complex, therefore each of its
vertices contributes in both sums $\sum\limits_{\sigma\in
F(K_P)}t^{|\sigma|}$ and $\sum\limits_{\sigma\in
K_P}t^{|\sigma|}$. Hence, the inequality for coefficients of
$\sum\limits_{\sigma\in F(K_P)}t^{|\sigma|}\leqslant
\sum\limits_{\sigma\in K_P}t^{|\sigma|}$ turns to an equality for
the constant term and $t$.
\end{proof}

\ex Let $P$ be a 3-dimensional polytope. Consider the following
numbers: $m$ --- number of facets, $e$ --- number of edges, $a_i$
--- number of vertices which are contained in exactly $i$ facets, where $i\geqslant
3$. Any edge is contained in exactly two facets. Then by
definition,
$$F_P(\alpha,t) = \alpha^3+m\alpha^2t+e\alpha t^2+a_3t^3+a_4t^4+a_5t^5+\ldots$$
Writing down the inequality of lemma \ref{lemIneqEq}, we get
$$1+mt+et^2+\sum\limits_{i\geqslant 3}a_it^i\leqslant -1+m(t+1)-e(t+1)^2+\sum\limits_{i\geqslant3}a_i(t+1)^i.$$
For the terms $t^0$ and $t^1$ there is an equality. For free term
we have
$$1=-1+m-e+\sum\limits_{i\geqslant3} a_i,$$
which is just an Euler formula. For $t^1$ there holds:
$$m = m-2e+\sum\limits_{i\geqslant3}ia_i;$$
$$2e = \sum\limits_{i\geqslant3}ia_i.$$
This can be seen directly from the edge graph of $P$.

\ex Consider a polytope $P$ of dimension $4$. Let $m$ be the
number of its facets, $r$ --- the number of its ridges (that is
faces of dimension $2$), $a_i$ --- the number of its edges which
are contained in exactly $i$ facets, and $b_i$ --- the number of
vertices of $P$ which are contained in exactly $i$ facets. Each
ridge is an intersection of two facets. Thus
$$F_P(\alpha,t) = \alpha^4+m\alpha^3t+r\alpha^2 t^2+\sum\limits_{i\geqslant 3}a_i\alpha t^i+\sum\limits_{i\geqslant 4}b_i t^i.$$
So the inequality of lemma \ref{lemIneqEq} has the form
\begin{multline*}
1+mt+rt^2+a_3t^3+\sum\limits_{i\geqslant 4}(a_i+b_i)t^i\leqslant\\
\leqslant 1 - m(t+1)+r(t+1)^2-a_3(t+1)^3+\sum\limits_{i\geqslant
4}(b_i-a_i)(t+1)^i.
\end{multline*}
For the free terms there holds an equality
$$1=1-m+r-\sum a_i+\sum b_i.$$
This is the Euler-Poincare formula as was expected. For the
$t^1$-terms the equality have the form
$$1+m = 1-m+2r-\sum\limits_{i\geqslant 3}ia_i+\sum\limits_{i\geqslant 4}ib_i;$$
\begin{equation}\label{equatTbayer}
2m-2r = \sum\limits_{i\geqslant
4}ib_i-\sum\limits_{i\geqslant 3}ia_i.
\end{equation}
The last equality can be rewritten in terms of flag $f$-numbers of
a polytope $P$. Indeed, $\sum ia_i$ is exactly the number of flags
$F\subset \F$, where $F$ is an edge and $\F$ is a facet of $P$. So
forth, $\sum ia_i = f_{\{1,3\}}$. Similarly, $\sum ib_i =
f_{\{0,3\}}$ and $2e = f_{\{2,3\}}$. Therefore, relation
\ref{equatTbayer} has the form
$$2f_{\{3\}} = f_{\{0,3\}}-f_{\{1,3\}}+f_{\{2,3\}}.$$
This is exactly one of the generalized Dehn-Sommerville relations,
discovered by M. Bayer and L. Billera \cite{BB}.

\begin{lemma}
Let $P$ be $n$-dimensional polytope, $f_{\{i\}}$ --- the number of
$i$-dimensional faces of $P$ and $f_{\{i,n-1\}}$ --- the number of
pairs $F\subset \F$, where $\dim F = i$ and $\dim \F = n-1$. Then
\begin{equation}\label{equatBB}
f_{\{n-1\}}(1+(-1)^n) =
\sum\limits_{i=0}^{n-2}(-1)^if_{\{i,n-1\}}.
\end{equation}
\end{lemma}

\begin{proof}
Let $f_{i,k}$ be the number of $i$-dimensional faces of $P$ which
are contained in exactly $k$ facets. Then by definition,
$F_P(\alpha,t) = \sum_{i,k}f_{i,k}\alpha^it^k$. In order to apply
lemma \ref{lemIneqEq} consider $t^0$ and $t^1$-terms of
polynomials $F_P(1,t)$ and $F_P(-1,t+1)$. We have:
$$
F_P(1,t) = A_0 + A_1t + O(t^2),
$$
where $A_0 = \sum_if_{i,0} = 1$, $A_1 = \sum_if_{i,1} =
f_{\{n-1\}}$, since the only face which is not contained in any
facet is $P$ itself and the only faces which are contained in
exactly $1$ facets are facets. The symbol $O(t^2)$ stands for
terms of degree greater than $1$. Proceeding further, we get
$$
F_P(-1,t+1) = \sum\limits_{i,k}f_{i,k}(-1)^i(t+1)^k = B_0 + B_1t +
O(t^2),
$$
where $$B_0 = \sum\limits_{i,k}f_{i,k}(-1)^i =
\sum\limits_{i}(-1)^if_{\{i\}},$$
$$
B_1 = \sum\limits_{i,k}f_{i,k}k(-1)^i =
\sum\limits_i(-1)^i\sum\limits_kf_{i,k}k =
\sum\limits_{i=0}^{n-2}(-1)^if_{\{i,n-1\}} +
(-1)^{n-1}f_{\{n-1\}}.
$$
Therefore, by lemma \ref{lemIneqEq} we get
$$A_0 = B_0,$$
$$1 = \sum\limits_{i}(-1)^if_{\{i\}},$$
which is Euler-Poincare relation for a polytope. We also have
$$A_1 = B_1,$$
$$f_{\{n-1\}} = \sum\limits_{i=0}^{n-2}(-1)^if_{\{i,n-1\}} + (-1)^{n-1}f_{\{n-1\}},$$
which is the required formula.
\end{proof}

\rem The equation \ref{equatBB} is exactly one of Bayer-Billera
relations \cite{BB}.

\rem The equalities of lemma \ref{lemIneqEq} can be reduced to the
relation on flag $f$-numbers. Nevertheless, the next example shows
that coefficients of two-dimensional face polynomial
$F_P(\alpha,t)$ can not be expressed in terms of flag $f$-numbers
in general.

\ex\label{examplNotflag} There exist two polytopes $P$ and $Q$
which have the same flag $f$-numbers, but $F_P(\alpha, t)\neq
F_Q(\alpha,t)$. Such polytopes $P$ and $Q=P^*$ are depicted in
figure \ref{polar}. Their flag numbers are the same: $f_{\{0\}}=7,
f_{\{1\}}=12, f_{\{2\}}=7, f_{\{0,1\}}=2f_{\{1\}}=24,
f_{\{0,2\}}=f_{\{1,2\}}=24$. But their two-dimensional face
polynomials are different: $F_P(\alpha,t) =
\alpha^3+7\alpha^2t+12\alpha t^2+4t^3+3t^4$, $F_{P^*}(\alpha,t) =
\alpha^3+7\alpha^2t+12\alpha t^2+5t^3+t^4+t^5$.

\section{Buchstaber number of a convex polytope}\label{Secbuch}
%

Recall that there are several definitions of the Buchstaber number
for different objects.

\begin{defin}\label{buchpoly}
The Buchstaber number $s(P)$ of a convex polytope $P$ is defined
as a maximal rank of toric subgroups of $T^m$ acting freely on
$\zz_P\subset \Co^m$.
\end{defin}

The definition \ref{buchpoly} generalizes the definition of $s(P)$
for simple polytopes.

\begin{defin}\label{buchhyper}
Let $G$ be an arbitrary hypergraph. Then a moment-angle complex
$\zz_G$ is defined. The torus $T^m$ acts on $\zz_G$
coordinatewise. The Buchstaber number $s(G)$ of a hypergraph is
defined as a maximal rank of toric subgroups of $T^m$ acting
freely on $\zz_G$.
\end{defin}

The definition \ref{buchhyper} generalizes the definition of
$s(K)$ for simplicial complexes.

\begin{defin}
Let $U$ be a complement to some coordinate space arrangement.
Then there is a coordinatewise action of torus on $U$. The
Buchstaber number $s(U)$ of a coordinate space arrangement is a
maximal rank of toric subgroups of $T^m$ acting freely on $U$.
\end{defin}

\begin{prop}\label{buchequality}
For a general convex polytope $P$ there holds: $$s(P) = s(G_P) =
s(K_P) = s(U_P)$$ in the sense of definitions above.
\end{prop}

\begin{proof}
Consider a general situation in which a torus $T^m$ acts on a
space $X$. Denote stabilizers of this action by $S_1,\ldots,S_l$.
A subgroup $G\subseteq T^m$ acts freely on $X$ iff $G\cap
S_i=\{\mathbf{1}\}\in T^m$ for any stabilizer $S_i$.

To prove that $s(P)=s(K_P)=s(U_P)$ we show that actions of $T^m$
on $\zp$, $\z_{K_P}$ and $U_{K_P}$ have the same stabilizers.
Indeed, in all three cases stabilizers are the subgroups
$G=T^{\{i_1,\ldots,i_k\}}$ where
$\F_{i_1}\cap\ldots\cap\F_{i_k}\neq\varnothing$. Such a subgroup
stabilizes a point over $x\in\F_{i_1}\cap\ldots\cap\F_{i_k}$ in
case of $\zp$. The same subgroup stabilizes a point
$(\varepsilon_1,\ldots,\varepsilon_m)\in(D^2,S^1)^{K_P}=\z_{K_P}\subseteq
U_{K_P}$, where $\varepsilon_j = 0$ for $j\in\{i_1,\ldots,i_k\}$
and $1$ otherwise. Since the sets of stabilizers are the same in
all three cases we have the same subgroups acting freely on
corresponding spaces.

The equality $s(G_P)=s(K_P)$ follows from $\z_{G_P}=\z_{K_P}$.
\end{proof}

\ex\label{examplPyramid} We may define a pyramid over arbitrary
convex polytope $Q$. The pyramid $\pyr Q$ is a convex hull of $Q$
and a point $v$, that does not lie in a plane of $Q$. A pyramid is
an operator on a set of all convex polytopes. Note that this
operator does not make sense on the set of all simple polytopes,
since $\pyr Q$ may be (and usually is) non-simple even if $Q$ is
simple.

There are many interesting properties of the operator $\pyr$. For
example, $(\pyr Q)^* = \pyr Q^*$. Therefore if $Q^*=Q$, then
$(\pyr Q)^* = \pyr Q$. This gives, in particular, that a pyramid
over a square on a figure \ref{pyr} is self-polar since square is
polar to itself.

Also we have $s(\pyr Q) = 1$ for any convex polytope $Q$. To prove
this consider all facets $\F_1,\ldots,\F_{m-1}$ of a $P=\pyr Q$
excluding its base. There $m$ is the number of all facets of $P$.
We have $\F_1\cap\ldots\cap\F_{m-1}\neq\varnothing$ since this
intersection is a cone point. Therefore there is a stabilizer
$T^{\{1,\ldots,m-1\}}\subset T^m$ of the action of the torus on
$\z_{K_P}$. The rank of this stabilizer is $m-1$. Since any
subgroup of $T^m$ acting freely on $\z_{K_P}$ should intersect
stabilizer only in the unit, the rank of such subgroup should be
less than or equal to $1$. But for any polytope $P$ there exists a
diagonal subgroup of a torus acting freely on $\zp$. Therefore
$s(P)$ is exactly $1$ in the case of pyramid.

We have a conjecture that pyramids are the only polytopes that
have a Buchstaber number equal to $1$. The similar fact for simple
polytopes is true: the only simple polytopes that have a
Buchstaber number $1$ are simplices (which are pyramids). This was
proved in \cite{Er}.


Since the Buchstaber number of a polytope is the same as
Buchstaber number of an appropriate simplicial complex we may
write some estimations according to
\cite{Izm},\cite{Er},\cite{Ay}.

\begin{prop}
Let $P$ and $Q$ be polytopes with $m_P$ and $m_Q$ facets
respectively. Then
$$m-\gamma(P)\leqslant s(P)\leqslant m-\dim(K_P)-1,$$
$$s(P)\leqslant m- \lceil\log_2(\gamma(P)+1)\rceil,$$
$$s(P\times Q)\geqslant s(P)+s(Q),$$
$$s(P\times Q)\leqslant \min(s(P)+m_Q-\dim(K_Q), s(Q)+m_P-\dim(K_P)).$$
\end{prop}

There $\gamma(P)$ is a minimal number of colors needed to paint
the facets of $P$ in such a way that intersecting facets have
different colors.

\rem For non-simple polytopes $P$ there exist stabilizers
$T^{\{i_1,\ldots,i_k\}}$ of dimension $k>n$. Therefore $s(P)<m-n$.

%

\section{Moment-angle complexes and Betti numbers}\label{Secmomang}

Let $\Zo[K]$ denote the a Stanley-Reisner ring of a simplicial
complex $K$ \cite{St2},\cite{BP}. By \cite{BG} the isomorphism
class of Stanley-Reisner ring defines the structure of $K$
uniquely. For a polytope $P$ define Stanley-Reisner ring $\Zo[P]$
by $\Zo[P] = \Zo[K_P]$. Then the isomorphism class of the ring
$\Zo[P]$ defines the combinatorial class of $P$ uniquely since $P$
can be restored from $K_P$.

We recall the definition of Betti numbers for simplicial
complexes. Details may be found in \cite{BP}, section 3.

\begin{defin}
Let $K$ be a simplicial complex on $m$ vertices. Set
$$\beta^{-i,2j}(K) =
\dim_{\Zo}\Tor^{-i,2j}_{\Zo[v_1,\ldots,v_m]}(\Zo[K],\Zo),\quad
0\leqslant i,j\leqslant m,$$ where

1) $\Zo[K]$ is a Stanley-Reisner ring of a complex $K$. The
natural projection $\Zo[v_1,\ldots,v_m]\to \Zo[K]$ defines a
structure of a $\Zo[v_1,\ldots,v_m]$-module on $\Zo[K]$.

2) A natural augmentation $\Zo[v_1,\ldots,v_m]\to \Zo$ defines a
structure of $\Zo[v_1,\ldots,v_m]$-module on $\Zo$.

3) The grading $-i$ corresponds to the term of resolution, and the
grading $2j$ corresponds to the natural grading in the ring of
polynomials.

Then numbers $\beta^{-i,2j}(K)$ are called Betti numbers of a
simplicial complex $K$.
\end{defin}

We have an isomorphism of graded algebras $H^*(\zk,\Zo)\cong
\Tor^{*}_{\Zo[v_1,\ldots,v_m]}(\Zo[K],\Zo)$, where the grading of
the $\Tor$-algebra is given by total degree. This isomorphism can
also be used to define a bigraded structure on $H^*(\zk,\Zo)$.

\begin{defin}
For a polytope $P$ define Betti numbers by
$\beta^{-i,2j}(P)=\beta^{-i,2j}(K_P)$.
\end{defin}

Hochster's formula can be written for general convex polytopes in
its usual form.

\begin{prop}[Hochster's theorem for convex
polytopes]\label{hochster}
Let $F_1,\ldots,F_m$ be a set of facets of a polytope $P$. Then
$$
\beta^{-i,2j}(P)=\sum\limits_{\omega\subseteq[m],
|\omega|=j}\dim_{\Zo} \tilde{H}^{j-i-1}(F_\omega;\Zo),
$$
where $F_\omega = \bigcup\limits_{i\in\omega} F_i\subseteq P$.
\end{prop}

\begin{proof}
By the definition
$\beta^{-i,2j}(P)=\beta^{-i,2j}(K_P)=\sum\limits_{\omega\subseteq[m],
|\omega|=j}\dim_{\Zo} \tilde{H}^{j-i-1}(K_\omega;\Zo),$ where
$K_w$ is a complete subcomplex of $K_P$ spanned on $\omega$.
Second equality represents Hochster's formula for simplicial
complexes. We claim that $K_\omega\simeq F_\omega$. Indeed, there
is a contractible cover of $F_\omega = \bigcup\limits_{i\in\omega}
F_i\subseteq P$ by subsets $F_i$. One can see, that the nerve of
this cover is $K_\omega$. Therefore $K_\omega\simeq F_\omega$ and
$\tilde{H}^{j-i-1}(K_\omega;\Zo)\cong\tilde{H}^{j-i-1}(F_\omega;\Zo)$
which concludes the proof.
\end{proof}

We have, by definition, $\dim
H^p(\z_{K_P},\Zo)=\sum\limits_{-i+2j=p}\beta^{-i,2j}(P)$. The next
statement shows that we can substitute $\z_{K_P}$ by $\zp$ in this
formula.

\begin{theorem}\label{theoremEquiv}
For any convex polytope $P$ there holds $$\z_{K_P}\simeq \zp.$$
\end{theorem}

\rem For non-simple $P$ the spaces $\zp$ and $\z_{K_P}$ are not
homeomorphic. This can be seen by dimensional reasons. If a
polytope $P$ has dimension $n$ and is not simple, then $K_P$
contains a simplex $\tau$ such that $|\tau|>n$. Therefore
$\dim\z_{K_P}=\dim(D^2,S^1)^K>m+n$. But $\dim\zp=m+n$ since $\zp$
can be described as an identification space of $P\times T^m$ (see
the proof below).

\begin{proof}
To prove this we use properties of homotopy colimits. First of all
recall that a space $\z_{K_P}$ can be described as a usual colimit
of a diagram. Let $\cat(K_P)$ be a small category associated to
simplicial complex $K_P$. The objects of $\cat(K_P)$ are simplices
of $K_P$. The morphisms of $\cat(K_P)$ are inclusions of
simplices. This means that there exists exactly one morphism from
$\sigma$ to $\tau$ whenever $\sigma\subseteq\tau$. In such terms
the empty simplex is the initial object of $\cat(K_P)$.

The diagram $\bar{D}_{K_P}\colon \cat(K_P)\to\Top$ is defined by
$\bar{D}_{K_P}(\sigma)=(D^2)^\sigma\times T^{[m]\setminus\sigma}$
and $\bar{D}_{K_P}(\sigma\hookrightarrow\tau)=((D^2)^\sigma\times
T^{[m]\setminus\sigma} \hookrightarrow (D^2)^\tau\times
T^{[m]\setminus\tau})$. Then $\colim \bar{D}_{K_P} \cong \z_{K_P}$
(see details in \cite{PR}).

Further we will use a description of $\zp$ using colimit.

1) The first step is to realize $\zp$ as a quotient space
$\zp\cong P\times T^m / \sim$, where $(x_1,t_1)\sim(x_2,t_2)$
whenever $x_1=x_2$ and $t_1t_2^{-1}\in T^{\tilde{\sigma}(x)}$. The
function $\tilde{\sigma}$ taking values in $2^{[m]}$ was
introduced in section \ref{Secdefinit}. Its value in point $x$ is
the set of all facets containing $x$.

To prove that the quotient space is homeomorphic to $\zp$ consider
an inclusion map $i_P = i\circ j_P\colon P\to \Ro^m\hookrightarrow
\Co^m$ which is a section of the map $p\colon \zp\to P$. Then a
map $h\colon P\times T^m\to \zp$, $h(p,t) = t\cdot i_P(x)$ induces
a required homeomorphism from $\zp\cong P\times T^m / \sim$.

2) We now realize a quotient space as a homotopy colimit.

Consider a small category $\cat(P)$. Its objects are all faces of
a polytope $P$ (here we consider the whole polytope as a face of
itself). The morphisms are given by inverse inclusions. This means
that there exists exactly one morphism from $F$ to $G$ whenever
$F\supseteq G$. The function $\tilde{\sigma}$ can be treated as a
functor between small categories $\tilde{\sigma}\colon
\cat(P)\to\cat(K_P)$. Indeed, any face $F$ defines a hyperedge in
$G_P$ thus a simplex in $K_P$. When $F\supseteq G$ we have
$\tilde{\sigma}(F)\subseteq\tilde{\sigma}(G)$.

Define a functor $D_P\colon \cat(P)\to \Top$ by the rules $D_P(F)
= T^{[m]\setminus\tilde{\sigma}(F)}$ and $D_P(F\supseteq G) =
q_{F,G}\colon T^{[m]\setminus\tilde{\sigma}(F)}\to
T^{[m]\setminus\tilde{\sigma}(F)}/T^{\tilde{\sigma}(G)\setminus
\tilde{\sigma}(F)}\cong T^{[m]\setminus\tilde{\sigma}(G)}$ --- a
natural projection to the quotient group.

We claim that a homotopy colimit of a diagram $D_P$ is exactly the
space $P\times T^m / \sim$. By definition, the homotopy colimit
$\hocolim D_P$ is a quotient of $\coprod\limits_{F\in\Ob\cat(P)}
|F\downarrow\cat(P)|\times D_P(F)$ under identifications $(x,
D_P(F\supseteq G)(y))\sim(x,y)$ for $x\in|G\downarrow
\cat(P)|\subseteq|F\downarrow\cat(P)|$ and $y\in D_P(F)$. In our
case the geometrical realization of an undercategory
$|F\downarrow\cat(P)|$ is a barycentric subdivision of $F$ itself.
Moreover inclusions of undercategories coincide with inclusions of
faces in $P$. Thus $\hocolim D_P$ is a quotient of
$\coprod\limits_{F\in\Ob\cat(P)}F\times
T^{[m]\setminus\tilde{\sigma}(F)}$ under identifications $(x,t)\sim
(x,q_{F,G}(t))$ for $x\in G\subseteq F$ and $t\in
T^{\tilde{\sigma}(F)}$. Such a space is obviously homeomorphic to
$P\times T^m / \sim$ introduced earlier.

3) We already have all ingredients to complete the proof of the
theorem.

Consider an ancillary diagram $\bar{D}_P\colon\cat(P)\to \Top$.
Let $\bar{D}_P(F) = T^{[m]\setminus\tilde{\sigma}(F)}\times
(D^2)^{\tilde{\sigma(F)}}$ and $\bar{D}(F\supseteq G)$ is the
inclusion of $T^{[m]\setminus\tilde{\sigma}(F)}\times
(D^2)^{\tilde{\sigma}(F)}$ into
$T^{[m]\setminus\tilde{\sigma}(G)}\times
(D^2)^{\tilde{\sigma}(G)}$. All the maps in the diagram
$\bar{D}_P$ are homotopically equivalent to the corresponding maps
in the diagram $D_P$. Therefore we have
$$
\hocolim D_P \simeq \hocolim \bar{D}_P \simeq \colim \bar{D}_P.
$$
The second equivalence holds since all the maps in $\bar{D}_P$ are
cofibrations.

Note that $\bar{D}_P = D_{K_P}\circ\tilde{\sigma}\colon \cat(P)\to
\Top$. So $\bar{D}_P$ may be treated as a subdiagram of
$\cat{K_P}$. It can be directly shown that these two diagrams have
the same colimit equal to $\bigcup\limits_{F\in
P}(D^2)^{\tilde{\sigma}(F)}\times T^{m\setminus\tilde{\sigma}(F)}=
(D^2,S^1)^{K_P}$. This concludes the proof.

\end{proof}

The theorem states that there is an additional homotopy
equivalence $\zp\simeq\z_{K_P}$ in a scheme \ref{mainscheme}. This
theorem together with proposition \ref{buchequality} lets us think
that a complex $K_P$ is a nice substitute for a convex polytope
$P$.

\ex\label{examplRestrict} Consider a pyramid $P$ over a square
(figure \ref{pyr}). As a consequence of a proof of theorem
\ref{theoremEquiv} we have $\zp \cong (P\times T^5)/\sim$.
Restricting this construction to the first facet $\F_1 = \triangle
ABC$ (which is one of side facets), we get the space
$p^{-1}(\F_1)\cong\F_1\times T^4/\sim$ where $T^4$ stands for the
product of circles corresponding to facets $\F_2,\F_3,\F_4,\F_5$.
The equivalence relation in this formula is given by

$$(x,t)\sim(y,s)\Leftrightarrow\begin{cases}
x=y\in(AB), ts^{-1}\in T^{\{5\}},\\
x=y\in(BC), ts^{-1}\in T^{\{2\}},\\
x=y\in(AC), ts^{-1}\in T^{\{4\}},\\
x=y=A, ts^{-1}\in T^{\{4,5\}},\\
x=y=B, ts^{-1}\in T^{\{2,5\}},\\
x=y=C, ts^{-1}\in T^{\{2,3,4\}}.
\end{cases}
$$

We see that the restriction of the moment-angle map $p$ over a
facet $\F_1$ does not coincide with $\z_{F_1}\times T^1$. It can
be deduced from the fact that $p^{-1}(\F_1)$ is simply connected:
all basic cycles of a torus can be contracted using the
equivalence relation. But $\z_{F_1}\times T^1$ is not simply
connected. This example shows that a restriction of a moment-angle
map over the facet is not so well behaved as in the case of simple
polytopes.
\\
\\
We hope that invariants of $P$ defined via the complex $K_P$ will
help to solve some problems. For example, the problem
\ref{problbuch} may lead to better understanding of self-polar
polytopes.

Instead of dealing with cohomology ring $H^*(\zz_{K_P})$ we can
look only at Betti numbers of a complex $K_P$.

\ex The only simple polytopes which are combinatorially equivalent
to their polar polytopes are simplices $\Delta^n$ in each
dimension and m-gons $P_m$ in dimension $2$. We have $s(\Delta^n)
= 1$ and $s(P_m) = m-2$.

\ex A pyramid $P$ shown on the figure \ref{pyr} is equivalent to
its polar. For such pyramid we have $s(P)=1$ and nonzero Betti
numbers of $K_P$ are $\beta^{0,0}(P)=1$, $\beta^{-1,6}(P)=2$,
$\beta^{-2,10}(P)=1$.

There is a nice way to calculate Betti numbers of pyramids using
Betti numbers of their bases.

\begin{prop}
For any convex polytope $Q$ there holds $\beta^{-i,2j}(\pyr Q) =
\beta^{-i,2(j-1)}(Q)$ for $j>0$.
\end{prop}

\rem The only nonzero Betti number with $j=0$ is
$\beta^{0,0}(K)=1$ for any simplicial complex $K$.

\begin{proof}
We use a Hochster formula \ref{hochster} for a polytope $\pyr Q$:
$$
\beta^{-i,2j}(\pyr Q)=\sum\limits_{\omega\subseteq[m],
|\omega|=j}\dim_{\Zo} \tilde{H}^{j-i-1}(F_\omega;\Zo),
$$
where $F_\omega = \bigcup\limits_{i\in\omega} \F_i\subseteq \pyr
Q$. We suppose that a base $Q$ of the pyramid is the first facet
in a given enumeration, $Q=\F_1$. If the set $\omega$ does not
contain $1$, then the set $F_\omega$ is contractible to the apex
of the pyramid. Therefore such terms do not contribute at the sum
at the right.


Suppose $\omega = \{1,i_2,\ldots,i_j\}$. Any facet $\F_i$ except
$Q$ is a pyramid over a facet $\tilde{\F}_i$ of a polytope $Q$.
Contracting $Q$ in $F_\omega$, we get $F_\omega\simeq
\Sigma(\tilde{\F}_{i_2}\cup\dots\cup\tilde{\F}_{i_j})$. The base
of a suspension is a space $\tilde{F}_{\omega\setminus\{1\}}$ for
a polytope $Q$. Then by suspension isomorphism in cohomology we
have
\begin{multline}
\sum\limits_{\omega\subseteq[m], |\omega|=j}\dim_{\Zo}
\tilde{H}^{j-i-1}(F_\omega;\Zo) = \sum\limits_{\omega\subseteq
[m]\setminus\{1\}, |\omega|=j-1}\dim_{\Zo}
\tilde{H}^{j-i-1}(\Sigma \tilde{F}_{\omega\setminus\{1\}};\Zo) =\\
=\sum\limits_{\omega\subseteq [m]\setminus\{1\},
|\omega|=j-1}\dim_{\Zo} \tilde{H}^{j-i-2}(
\tilde{F}_{\omega\setminus\{1\}};\Zo) = \beta^{-i,2(j-1)}(Q).
\end{multline}
This concludes the proof.
\end{proof}

Another statement is a trivial consequence of the multiplicativity
of Betti numbers for simplicial complexes and proposition
\ref{crossjoin}.

\begin{prop}
For convex polytopes $P$ and $Q$ there holds an equality
$$
\beta^{-i,2j}(P\times Q) = \sum\limits_{\substack{i_1+i_2=i\\
j_1+j_2=j}}\beta^{-i_1,2j_1}(P)\beta^{-i_2,2j_2}(Q).
$$
\end{prop}

Using Betti numbers and Buchstaber number, we can find an
obstruction for a polytope to be polar to itself. If $P = P^*$,
then first of all $f_i(P) = f_{n-i}(P)$, where $f_i(P)$ is the
number of $i$-faces of $P$. This means that $f$-vector
$(f_0,\ldots,f_n)$ is symmetric, which provides a necessary
condition for a polytope to be self-polar. We also have
$K_P=K_{P^*}$ therefore $\beta^{-i,2j}(P)=\beta^{-i,2j}(P^*)$ and
$s(P) = s(P^*)$.

\begin{figure}[h]
\begin{center}
\includegraphics[scale=0.3]{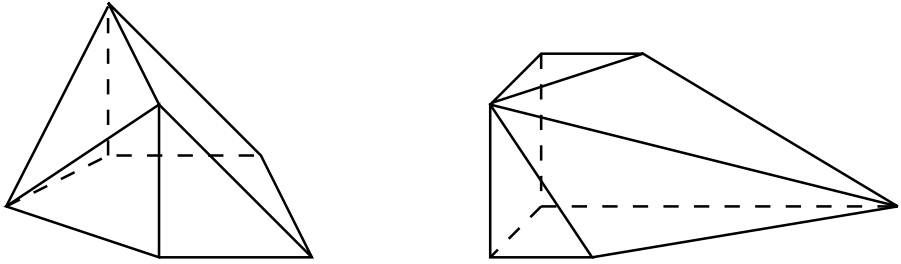}
\end{center}
\caption{Polytope $P$ and its polar $P^*$}\label{polar}
\end{figure}

\ex Consider two polytopes $P$ and $P^*$ depicted on figure
\ref{polar}. One can see that they are not equivalent (this was
discussed in example \ref{examplNotflag}). But we cannot
distinguish these polytopes using only their $f$-vectors: they
both are equal to $(7,12,7)$. Nevertheless, they have different
Betti numbers. Nonzero Betti numbers of $P$ are
$\beta^{0,0}(P)=1$, $\beta^{-1,4}(P)=3$, $\beta^{-1,6}(P)=6$,
$\beta^{-2,8}(P)=14$, $\beta^{-3,10}(P)=9$, $\beta^{-4,14}(P)=1$.
Nonzero Betti numbers of $P^*$ are $\beta^{0,0}(P^*)=1$,
$\beta^{-1,4}(P^*)=2$, $\beta^{-1,6}(P^*)=6$,
$\beta^{-2,8}(P^*)=15$, $\beta^{-3,10}(P^*)=9$,
$\beta^{-4,14}(P^*)=1$. This example shows that Betti numbers are
more strong invariants than $f$-vector. The computation was made
using Macaulay2 software.



\end{document}